\documentclass[a4paper,11pt]{article}
\usepackage[utf8]{inputenc}
\usepackage{amsmath,amsfonts}
\usepackage{amsthm}
\usepackage{amssymb}
\usepackage[parfill]{parskip}
\usepackage[english]{babel}
\usepackage{amsfonts}
\usepackage[T1]{fontenc}
\usepackage[utf8]{inputenc}
\usepackage{enumerate}
\usepackage{verbatim}
\usepackage{graphicx}
\usepackage{verbatim}
\usepackage[all]{xy}
\usepackage{tikz}
\usepackage{relsize}
\usepackage{amsmath}
\usepackage{amscd}
\usepackage{hyperref}
\usepackage{caption}
\usepackage{xcolor}

\usepackage{textcomp}
\makeatletter
\def\thm@space@setup{%
  \thm@preskip=\parskip \thm@postskip=0pt
}
\makeatother
\newtheorem{teo}{Theorem}[section]
\newtheorem{lem}[teo]{Lemma}
\newtheorem{prop}[teo]{Proposition}
\newtheorem{cor}[teo]{Corollary}
\newtheorem{que}{Question}
\newtheorem*{teo*}{Theorem}

\newtheorem{cnj}{Conjecture}
\newtheorem*{prop:bound}{Theorem \ref{prop:bound}}
\newtheorem*{teo:cong}{Theorem \ref{main}}
\newtheorem*{lem:stima1}{Lemma \ref{stima1}}
\newtheorem*{lem:stima3}{Lemma \ref{stima3}}
\newtheorem*{lem:symm}{Lemma \ref{symm}}
\newtheorem*{teo:AMU}{Conjecture \ref{AMU}}
\newtheorem*{teo:moreAMU}{Theorem \ref{moreAMU}}
\newtheorem*{cor:upper}{Corollary \ref{upper}}
\theoremstyle{definition}
\newtheorem{dfn}[teo]{Definition}
\newtheorem*{dfn*}{Definition}

\newtheorem{oss}[teo]{Remark}
\newtheorem*{dom*}{Question}

\newcommand{\R}{\mathbb{R}}
\newcommand{\Z}{\mathbb{Z}}

\newcommand\abs[1]{\left|#1\right|}
\newcommand{\C}{\mathbb{C}}

\newcommand{\ra}{\rightarrow}
\newcommand{\li}{\textrm{Li}_2}

\newcommand{\lmin}{{\ell_{\rm min}}}
\newcommand{\vol}{{\rm vol}}

\newcommand{\bdy}{{\partial}}

\title{Growth of quantum $6j$-symbols and applications to the Volume Conjecture}
\author{Giulio Belletti, Renaud Detcherry, Efstratia Kalfagianni\footnote{E.K. is supported by NSF Grants DMS-1404754 and DMS-1708249 and a grant from the Institute for Advanced Study School of Mathematics}  \& Tian Yang\footnote{T.Y. is supported by NSF Grant DMS-1812008}}
 \date{}

\begin{document}
\maketitle

\begin{abstract}
We prove the Turaev-Viro  invariants volume conjecture for a ``universal''  class of cusped  hyperbolic 3-manifolds
that produces all 3-manifolds with empty or toroidal boundary by Dehn filling.
This leads to two-sided bounds  on the volume of any hyperbolic 3-manifold with empty or toroidal boundary
in terms of the growth rate of the Turaev-Viro invariants of the complement of an appropriate link contained in the manifold.
We also  provide evidence for a conjecture of Andersen, Masbaum and Ueno (AMU conjecture) about certain  quantum representations of
 surface mapping class groups.
 
A key step in our proofs
is finding a sharp upper bound on the growth rate of the quantum $6j-$symbol evaluated at $q=e^{\frac{2\pi i}{r}}.$ 
\end{abstract}


\section{Introduction}

The Turaev-Viro invariants $TV_r(M,q)$ of a compact $3$-manifold $M$ \cite{TuraevViro} are real numbers  depending on 
an integer  $r\geqslant 3,$ called level, and a $2r$-th root of unity $q.$  It has been long   known  that when one chooses $q=e^{\frac{\pi i}{r}},$ the invariants $TV_r(M,q)$ grow at most polynomially in $r.$
In contrast to that, in  \cite{Chen-Yang}, Chen and Yang's extensive computation
of the case $q=e^{\frac{2\pi i}{r}}$ suggests that for hyperbolic manifolds the growth is instead exponential and determines the volume.  They  stated  the following.

\begin{cnj}\label{volconj}
 Let $M$ be a hyperbolic $3$-manifold, either closed, with cusps, or compact with totally geodesic boundary. Then as $r$ varies along the odd natural numbers,
 $$
 \lim_{r\to \infty} \frac{2\pi}{r}\log\left|TV\left(M,e^{\frac{2\pi i}{r}}\right)\right|    =\textrm{Vol}(M).
 $$
 
\end{cnj} 

The main result in this article is to verify Conjecture \ref{volconj} for complements of fundamental shadow links.
These links, first considered by  Costantino and D. Thurston \cite{CosThurston}, are an infinite family of hyperbolic  links in connected sums $\#^{c+1}(S^1\times S^2),$ for any $c>0,$ with the following properties.

\begin{enumerate}
 \item The volume of any fundamental shadow link in  $\#^{c+1}(S^1\times S^2)$  is  equal to $2cv_8,$  where $v_8\cong 3.66$ is the volume of the regular ideal hyperbolic octahedron.  
 \item The Reshetikhin-Turaev invariants on any  fundamental shadow link have simple formulae (see Lemma \ref{shform}).
 \item The links form a universal class, in the sense that any orientable $3$-manifold with empty or toroidal  boundary is obtained from  a complement of a fundamental shadow by Dehn filling.
 \end{enumerate}
The main result of this article is the following.

\begin{teo}\label{main} For any $c>0,$ Conjecture \ref{volconj} holds for the complement of any  fundamental shadow link $L$ in  $\#^{c+1}(S^1\times S^2).$
\end{teo}

Detcherry, Kalfagianni and Yang  \cite{DKY} verified
Conjecture \ref{volconj}   for the complements of the Borromean ring and of the figure eight knot, and Ohtsuki  \cite{Ohtsuki} verified it for all hyperbolic 3-manifolds obtained by integral Dehn surgeries on the figure eight knot.  In \cite{Ka-Ho1} Wong verified Conjecture \ref{volconj} for certain octahedral links in $S^3$ called Whitehead chains and Belletti \cite{BellettI} proved it for complements of families of octahedral links
in connected sums copies of $S^1\times S^2$ that are not treated by Theorem \ref{volconj}.
Furthermore, Conjecture \ref{volconj} has been generalized to assert that the Turaev-Viro invariants determine the Gromov norm of  any compact orientable  3-manifold \cite{DK}. 
The generalized conjecture was proven for all  Gromov norm zero links
in 3-manifolds that are connected sums of  $S^3$ with copies of $S^1\times S^2$ by Detcherry and Kalfagianni \cite{DK}. An alternative proof for the case of Gromov norm zero knots in $S^3$ was given in \cite{DKY}.
The conjecture was also 
 shown to be closed under certain link cabling and satellite operations \cite{D, DK, Ka-Ho2}.

In \cite{C:volconj}   Costantino proved an extension of Kashaev's original volume conjecture \cite{kashaev:vol-conj} for fundamental shadow links.
His approach was to consider a version of the colored Jones polynomials of links in 
manifolds of the form $\#^{c+1}(S^1\times S^2),$ and for fundamental shadow links relate its asymptotics to the volume of their complement.

The basic building block in the definition of the Turaev-Viro invariants
 is the quantum $6j$-symbol. 
We  will recall the definitions and basic properties in Section \ref{sec6j}.
A  key ingredient in our proof  of Theorem \ref{main} is the following
upper bound on the growth of the quantum $6j$-symbol.

\begin{teo}\label{prop:bound}
For any $r$ and any $r$-admissible $6$-tuple $(n_1,n_2,n_3,n_4,n_5,n_6),$ we have
$$ \frac{2 \pi}{r}\log \abs{\Bigg|\begin{matrix}
   n_1 &n_2&n_3\\
   n_4&n_5&n_6
  \end{matrix}\Bigg|_{q=e^{\frac{2\pi i}{r}}}} \leqslant v_8+O\left(\frac{\log(r)}{r}\right).$$
\end{teo}

 Theorem \ref{prop:bound} should be compared with Costantino's result (Theorem \ref{Vol}) that under certain constraints, the exponential growth rate of a sequence of quantum $6j$-symbols coincides with the volume of a hyperbolic truncated tetrahedron whose dihedral angles are determined by the sequence of $6$-tuples.  The inequality of Theorem \ref{prop:bound} is sharp in the sense that for some sequences the limit as $r\to \infty$ is $v_8$ (Lemma \ref{r/2}).

Combining Theorem \ref{main} with a result of Futer, Kalfagianni and Purcell \cite{fkp:filling}, we show that the volume of any hyperbolic 3-manifold $M,$ with empty or toroidal boundary,  is estimated 
in terms of the Turaev-Viro invariants of an appropriate link contained in $M$ and that the estimate is asymptotically sharp.

To state our result, given a hyperbolic 3-manifold $N$ containing $k$  embedded horocusps  choose a slope $s_i$ on the boundary torus of each of them,
  and let 
$l_{\text{min}}$  denote the shortest length of any of the $s_i.$ We write $M=N(s_1, \dots, s_k)$ for the 3-manifold obtained  by Dehn filling  $N$ along these $k$ slopes. 
Also as $r$ varies along the odd natural numbers let 
 $$lTV(N)=\liminf_{r\to\infty} \frac{2\pi}{r}\log\left|TV\left(M,e^{\frac{2\pi i}{r}}\right)\right| $$
and $$  LTV(N)= \limsup_{r\to\infty} \frac{2\pi}{r}\log\left|TV\left(M,e^{\frac{2\pi i}{r}}\right)\right|.$$
\vskip 0.03in

\begin{teo}\label{Dehnvol} 
Let  $M$ be a hyperbolic $3$--manifold possibly with cusps. There exists a cusped hyperbolic 3-manifold $N$ with $M=N(s_1, \dots, s_k),$  for some $k\geqslant1,$
and such that $lTV(N)=LTV(N)=\vol(N)$, and
$$ \alpha(\lmin) \  lTV(N) \leqslant  \vol(M) < lTV(N).$$
Here $0\leqslant \alpha(\lmin)\leqslant 1$ is an explicit function and $ \alpha(\lmin)$ approaches 1 as  as $\lmin\to \infty.$
\end{teo}

Theorem \ref{main} also has application to a conjecture of Andersen, Masbaum and Ueno  about the quantum representations of surface mapping class groups (AMU conjecture) \cite{AMU}. 
For a compact  orientable surface of genus $g$ and $n$ boundary components
$\Sigma_{g, n},$ let $\mathrm{Mod}(\Sigma_{g, n})$ denote its  mapping class group.
The AMU conjecture asserts that  the $SU(2)$ and $SO(3)$ quantum
representations of $\mathrm{Mod}(\Sigma_{g, n})$ send mapping classes with non-trivial  pseudo-Anosov pieces to elements of infinite order (for large enough level). The reader is referred
to Section \ref{secamu} for more details. Mapping classes are realized as monodromies of fibered 3-manifolds and, in particular, mapping classes of surfaces with boundary are realized as monodromies of complements of fibered links in 3-manifolds. It has been long known that
fibered links, and in particular hyperbolic fibered links,  exist in all orientable 3-manifolds with empty or toroidal boundary.
Here we  prove the following.

 \begin{teo} \label{moreAMU}  Let $M$ be the complement of a fundamental shadow link or the double of such a manifold. Given any link $L$ in $M$  there is an additional  knot $K\subset M$
  such that the complement  $M\setminus (K\cup L)$ fibers over $S^1$ with fiber a surface. Moreover, any monodromy of a fibration  of  $M\setminus (K\cup L)$ satisfies the AMU conjecture.
 \end{teo}
 
 Theorem \ref{moreAMU} uses a result of Detcherry and Kalfagianni \cite{DK, DK:AMU}  which shows that monodromies of a fibered 3-manifold $M$
 satisfy the  AMU conjecture provided that $lTV(M)>0;$ that is, provided that the Turaev-Viro invariants grow exponentially with respect to  $r.$
In \cite{DK:AMU} the authors used the handful of examples of  link complements in $S^3$ with $lTV(S^3\setminus L)>0$ known at the time, to construct
  the first infinite families of examples that satisfy the  AMU conjecture in surfaces $\Sigma_{g,n}$ with $g\geqslant 2$ and $n \geqslant 2.$
Since the class of fundamental shadow links is universal, 
 Theorem \ref{moreAMU} provides an abundance of fibered 3-manifolds with monodromies satisfying the AMU conjecture. Explicit constructions of such manifolds are given in
 \cite{DK:cosets}.

The paper is organized as follows: We recall the quantum $6j$-symbols and preliminaries 
 about Turaev-Viro invariants in Section \ref{sec6j}. In Section \ref{secbound} we prove the upper bound given in Theorem \ref{prop:bound}; the proofs of the technical lemmas used are postponed to Section \ref{sectec}. In Section \ref{secvol} we introduce fundamental shadow links and prove Theorem \ref{main}. Applications of the main result on the AMU conjecture and the volume comparison are included in section \ref{FSL}. We also include a proof of Costantino's result, originally proved for the root    of unity $q=e^{\frac{\pi i}{r}}$ in \cite{C},  at a different root of unity $q=e^{\frac{2\pi i}{r}}$ in the Appendix.

All the 3-manifolds we will consider in this paper will be orientable with empty or toroidal boundary.

\textbf{Acknowledgements.} Belletti wishes to warmly thank his advisors, Francesco Costantino and Bruno Martelli, for their guidance and advice, and Roland van der Veen for helpful conversations. 
Part of the research of this paper was done while he participated at  the thematic semester ``Invariants in low-dimensional geometry and topology'' organized by the CIMI labex in Toulouse in 2017.
Detcherry, Kalfagianni and Yang made progress on this paper while attending the BANFF workshop ``Modular Forms and Quantum invariants'' in March 2018. The authors thank the organizers of these events for invitations and for arranging excellent working conditions.


\section{The quantum \texorpdfstring{$6j$}f-symbols}\label{sec6j}

In this section we give the basic definitions relating Turaev-Viro invariants and quantum $6j$-symbols. Throughout
the rest of the paper $r\geqslant 3$ is an odd integer and $q=e^{\frac{2\pi i}{r}}.$ The \emph{quantum integer}
$\{n\}$ is defined as $q^n-q^{-n},$ and the quantum factorial $\{n\}!$ is $\prod_{i=1}^n\{i\}.$ Furthermore,
we denote with $I_r$ the set $\{0,1,\dots,r-2\}.$

\begin{oss}
In the remainder of the paper, we deal with the $SU(2)$ version of the Turaev-Viro invariants; however, everything remains true for the $SO(3)$ version, with small modifications.
\end{oss}

\begin{dfn}\label{adm}
 We say that a triple $(a,b,c)$ of non-negative integers is \emph{$r$-admissible} if
 \begin{itemize}
  \item $a,b,c\leqslant r-2;$
  \item $a+b+c$ is even and $\leqslant 2r-4;$
  \item $a\leqslant b+c,$ $b\leqslant a+c$ and $c\leqslant a+b.$ 
 \end{itemize}
We say that a $6$-tuple $(n_1,n_2,n_3,n_4,n_5,n_6)$ is $r$-admissible if the $4$ triples $(n_1,n_2,n_3),$
$(n_1,n_5,n_6),$ $(n_2,n_4,n_6)$ and $(n_3,n_4,n_5)$ are $r$-admissible.
\end{dfn}

Notice that, while in part of the literature, e.g. \cite{TuraevViro} or \cite{Chen-Yang}, the colors are half integers, we take them to be integers. Our notation will be very similar to that of \cite{TuraevViro}, except for the integer colors, and the use of $\{n\}$ instead of $[n]:=\frac{\{n\}}{\{1\}}.$ This will account for an extra $\{1\}$ factor in some of our formulas. We follow closely the notation of \cite{DK}.

For an $r$-admissible triple $(a,b,c)$ we can define
\begin{equation*}
 \Delta(a,b,c)=\left(\sqrt{-1}\zeta_r\frac{\{\frac{a+b-c}{2}\}!\{\frac{a-b+c}{2}\}!\{\frac{-a+b+c}{2}\}!}{\{\frac{a+b+c}{2}+1\}!}\right)^{\frac{1}{2}}
\end{equation*}
where $\zeta_r=2 \sin\left(\frac{2 \pi}{r}\right)=-\sqrt{-1}\{1\}_{|q=\exp(2\pi \sqrt{-1}/r)}.$ Notice that the number inside the square root is real: each $\{n\}$ is a purely imaginary number, and all the $\sqrt{-1}$ simplify. By convention we take the positive square root of a positive number, and the square root with positive imaginary part of a negative number.

Moreover, for an $r$-admissible $6$-tuple $(n_1,n_2,n_3,n_4,n_5,n_6)$ we can define its quantum $6j$-symbol as

\begin{gather}
 \begin{vmatrix}
   n_1 &n_2&n_3\\
   n_4&n_5&n_6
  \end{vmatrix}=\zeta_r^{-1}\left(\sqrt{-1}\right)^\lambda\prod_{i=1}^4\Delta(v_i)
 \sum_{z=max (T_i)}^{min (Q_j)}\frac{(-1)^z\{z+1\}!}{\prod_{i=1}^4\{z-T_i\}!\prod_{j=1}^3\{Q_j-z\}!}\label{sixj}
\end{gather}
where:
\begin{itemize}
 \item $\lambda=\sum_{i=1}^6 n_i;$
 \item $v_1=(n_1,n_2,n_3),$ $v_2=(n_1,n_5,n_6),$ $v_3=(n_2,n_4,n_6)$ and $v_4=(n_3,n_4,n_5);$
 \item $T_1=\frac{n_1+n_2+n_3}{2},$ $T_2=\frac{n_1+n_5+n_6}{2},$ $T_3=\frac{n_2+n_4+n_6}{2}$ and $T_4=\frac{n_3+n_4+n_5}{2};$
 \item $Q_1=\frac{n_1+n_2+n_4+n_5}{2},$ $Q_2=\frac{n_1+n_3+n_4+n_6}{2}$ and $Q_3=\frac{n_2+n_3+n_5+n_6}{2}.$
\end{itemize}
\begin{oss}
 Notice that if $z\geqslant r-1,$ then the summand in \eqref{sixj} corresponding to $z$ is equal to $0.$
\end{oss}

\begin{dfn}

 An \emph{$r$-admissible coloring} for a tetrahedron $T$ is an assignment of an $r$-admissible $6$-tuple $(n_1,n_2,n_3,n_4,n_5,n_6)$ to the set of edges of $T,$ as
 shown in figure \ref{fig:colamm}. Similarly, we define an \emph{$r$-admissible coloring} of a triangulation of a $3$-manifold, as an assignment of elements of $I_r$ to each of its edges, in such a way
 that the $6$-tuple assigned to the edges of each tetrahedron is an $r$-admissible coloring.
\end{dfn}

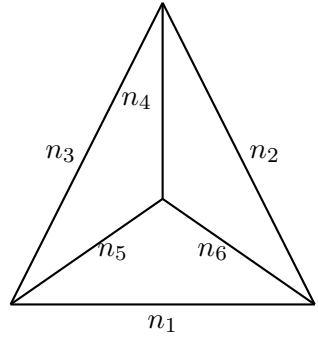
\begin{figure}
\centering
 \begin{tikzpicture}
 \draw [thick] (0,0)-- node[below] {$n_1$}(4,0);
 \draw [thick] (0,0)-- node[above, left] {$n_3$}(2,4);
 \draw [thick] (2,4)-- node[above,right] {$n_2$}(4,0);
 \draw [thick] (0,0)-- node[below,right] {$n_5$}(2,1.4);
 \draw [thick] (2,1.4)-- node[left] {$n_4$} (2,4);
 \draw [thick] (2,1.4)-- node [below,left] {$n_6$}(4,0);
\end{tikzpicture}
\caption{An admissible coloring for a tetrahedron}\label{fig:colamm}
\end{figure}

Let $M$ be an orientable compact $3$-manifold with a partially ideal triangulation $\tau.$ By this we mean
that some vertices of the triangulation are truncated, and the truncated faces are a triangulation for $\partial M.$

Denote with $A_r(\tau)$ the set of $r$-admissible colorings of $\tau,$ with $V$ the set of interior vertices of $\tau$ and with
$E$ the set of interior edges (by which we mean edges that are not contained in the boundary). If
$col\in A_r(\tau)$ and $T\in\tau$ we denote with $\lvert T\rvert_{col}$ the quantum $6j$-symbol corresponding to the $6$-tuple
that $col$ assigns to the edges of $T.$ Similarly, if $e\in E$ we define $$\lvert e\rvert_{col}=(-1)^{col(e)}\frac{\{col(e)+1\}}{\{1\}}.$$

Define the Turaev-Viro invariant of $M$ at level  $r$ in the root $q$ as
\begin{displaymath}
 TV_r(M,\tau,q):=\left(\frac{\sqrt{2}\sin\left(\frac{2\pi}{r}\right)}{\sqrt{r}}\right)^{2\lvert V\rvert}\sum_{col\in A_r(\tau)}\prod_{e\in E}\lvert e\rvert_{col}\prod_{T\in \tau}\lvert T\rvert_{col}.
\end{displaymath}

 By \cite{TuraevViro} if  $\tau$ and $\tau'$ are two partially ideal triangulations of $M,$ then $TV_r(M,\tau,q)=TV_r(M,\tau',q).$
Hence we have a topological invariant of $M,$ denoted by  $TV_r(M,\ q),$ depending on $r$ and $q.$


\section{The upper bound of the quantum \texorpdfstring{$6j$}f-symbol}\label{secbound}

In this section  and the next section we complete the proof of Theorem \ref{prop:bound}. In this section we give the proof assuming three technical lemmas the proofs of which  occupy Section \ref{sectec}.

Denote with $\Lambda(x)$ the Lobachevski function, defined as

\begin{displaymath}
 \Lambda(x):=-\int_0^x \log \lvert 2\sin(t)\rvert dt.
\end{displaymath}
It is $\pi$-periodic, odd, and real analytic outside of $\{k\pi, k\in \Z\}.$

 The tool used to estimate the quantum $6j$-symbol is the following lemma, first appeared in \cite[Proposition 8.2]{gale} for $q=e^{\frac{i\pi}{r}},$ and then in the other roots of unity in \cite[Proposition 4.1]{DK}.

\begin{lem}\label{stima}
 For any integer $0<n<r$ and at $q=e^{\frac{2\pi i}{r}},$
 \begin{displaymath}
  \log\left|\{n\}!\right|=-\frac{r}{2\pi}\Lambda\left(\frac{2n\pi}{r}\right)+O\left(\log(r)\right),
 \end{displaymath}
where the term $O(\log(r))$ is such that there exist constants $C,r_0$ independent of $n$ and $r$ such that 
$O(\log(r))\leqslant C \log(r)$ whenever $r>r_0.$
\end{lem}
\vskip 0.03in

\begin{oss}
 If $0<n<r-1,$ we can equally well use the estimate
 \begin{displaymath}
  \log\left|\{n+1\}!\right|=-\frac{r}{2\pi}\Lambda\left(\frac{2n\pi}{r}\right)+O(\log(r)),
 \end{displaymath}
since by applying a Taylor expansion to $\Lambda$ we find
\begin{align*}
 \Lambda\left(\frac{2n\pi}{r}+\frac{2\pi}{r}\right)-\Lambda\left(\frac{2n\pi}{r}\right)=&\frac{2\pi}{r}\Lambda'\left(\frac{2n\pi}{r}\right)+o\left(\frac{1}{r}\right)=\\=&-\frac{2\pi}{r}\log\left\rvert 2\sin\frac{2n\pi}{r}\right\lvert+o\left(\frac{1}{r}\right)= O\left(\frac{\log\left(r\right)}{r}\right)
\end{align*}
since $\left\lvert\sin\left(\frac{2\pi n}{r}\right)\right\rvert\geqslant \frac{\pi}{r}$ (because $2n\neq r$), and thus $-\frac{2\pi}{r}\log(\rvert 2\sin\frac{2n\pi}{r}\lvert)\leqslant \frac{\log(r)}{r},$ since $\log(ax)\leqslant a\log(x).$ Notice again that 
the constants involved in the $O\left(\frac{\log(r)}{r}\right)$ are independent of $n$ and $r.$
\end{oss}

We need the following  lemma.

\begin{lem}[\cite{DKY}, Lemma A.3]\label{sym}
For $i\in \{0,\dots,r-2\},$ let $i'=r-2-i.$ Then for any admissible $6$-tuple $(n_1,n_2,n_3,n_4,n_5,n_6),$ the $6$-tuples $(n_1,n_2,n_3,n_4',n_5',n_6')$ and $(n_1',n_2',n_3,n_4',n_5',n_6)$ are admissible and at $q=e^{\frac{2\pi i}{r}},$ 
\begin{equation*}
\begin{vmatrix}
   n_1 &n_2&n_3\\
   n_4&n_5&n_6
  \end{vmatrix}
=\begin{vmatrix}
   n_1 &n_2&n_3\\
   n_4'&n_5'&n_6'
  \end{vmatrix}
=\begin{vmatrix}
   n_1' &n_2'&n_3\\
   n_4'&n_5'&n_6
  \end{vmatrix}.\end{equation*}
\end{lem}

We are now ready to begin the proof of Theorem \ref{prop:bound}
stated in the Introduction. The proof of the theorem will be completed in Section \ref{sectec}.

\begin{prop:bound}
For any $r,$ and any admissible $6$-tuple $(n_1,n_2,n_3,n_4,n_5,n_6),$ then 
$$ \frac{2 \pi}{r}\log\abs{\Bigg|\begin{matrix}
   n_1 &n_2&n_3\\
   n_4&n_5&n_6
  \end{matrix}\Bigg|_{q=e^{\frac{2\pi i}{r}}}}\leqslant v_8+O\left(\frac{\log(r)}{r}\right)$$
  where $v_8\cong 3.66$ is the volume of the regular ideal hyperbolic octahedron.
\end{prop:bound}

\begin{proof}

Applying Lemma \ref{stima} (together with the subsequent remark) to the formula for the quantum $6j$-symbol \eqref{sixj} we obtain the estimate

\begin{gather}\label{estimate}
  \frac{2 \pi}{r}\log\abs{\Bigg|\begin{matrix}
   n_1 &n_2&n_3\\
   n_4&n_5&n_6
  \end{matrix}\Bigg|_{q=e^{\frac{2\pi i}{r}}}}\leqslant V(\theta_1,\theta_2,\theta_3,\theta_4,\theta_5,\theta_6)+O\left(\frac{\log(r)}{r}\right)
\end{gather}

where

\begin{gather}\begin{split}V(\theta_1,\theta_2,\theta_3,\theta_4,\theta_5,\theta_6):=\max_{\max (U_i)\leqslant Z \leqslant \min (V_j,2\pi)}F(Z,\theta_1,\theta_2,\theta_3,\theta_4,\theta_5,\theta_6)+\\ \nu(\theta_1,\theta_2,\theta_3)+
  \nu(\theta_1,\theta_5,\theta_6)+\nu(\theta_2,\theta_4,\theta_6)+\nu(\theta_3,\theta_4,\theta_5)\end{split}\end{gather} and we have
\begin{itemize}
 \item $\theta_i=\frac{2 \pi n_i}{r}$ and $Z=\frac{2 \pi z}{r};$
 \item $U_i=\frac{2 \pi T_i}{r}$ and similarly $V_j=\frac{2 \pi Q_j}{r};$

 \item $\nu(\alpha,\beta,\gamma)=\frac{1}{2}(\Lambda(\frac{\alpha+\beta+\gamma}{2})-\Lambda(\frac{\alpha+\beta-\gamma}{2})-\Lambda(\frac{\alpha-\beta+\gamma}{2})-\Lambda(\frac{-\alpha+\beta+\gamma}{2}));$
\item $F(Z,\theta_1,\theta_2,\theta_3,\theta_4,\theta_5,\theta_6)=\sum_{i=1}^4\Lambda(Z-U_i)+\sum_{j=1}^3\Lambda(V_j-Z)-\Lambda(Z).$
 \end{itemize}

The admissibility conditions imply that all the variables involved in above formulae take values between $0$ and $2\pi.$ 
Notice that the variables  $\theta_i$ satisfy similar triangular inequalities and admissibility conditions as the variables  $n_i.$
In particular $\theta_1+\theta_2+\theta_3\leqslant 4\pi,$ $\theta_1+\theta_5+\theta_6\leqslant 4\pi,$ $\theta_2+\theta_4+\theta_6\leqslant 4\pi$ and $\theta_3+\theta_4+\theta_5\leqslant4\pi.$

Next we  want to maximize $V$ subject to the admissibility conditions of the variables  $\theta_i.$ The argument relies   the three technical lemmas, whose proofs will occupy Section \ref{sectec}.
The first two lemmas are the following.
\begin{lem}\label{stima1}
 If 
 $0\leqslant\alpha,\beta,\gamma\leqslant \pi,$
then $\nu(\alpha,\beta,\gamma)\leqslant 0.$
\end{lem}

\begin{lem}\label{stima3}
 If $0\leqslant\theta_1,\theta_2,\theta_3,\theta_4,\theta_5,\theta_6\leqslant 2\pi$ and $max(T_i)\leqslant Z\leqslant min(Q_j,2\pi),$ then 
 \begin{displaymath}
  F(Z,\theta_1,\theta_2,\theta_3,\theta_4,\theta_5,\theta_6)+2\nu(\theta_1,\theta_2,\theta_3)\leqslant 8\Lambda\left(\frac{\pi}{4}\right)=v_8
 \end{displaymath}

\end{lem}

We obtain the following corollary.

\begin{cor}\label{stimacor}We have
 $$\max_{\max (U_i)\leqslant Z \leqslant \min (V_j,2\pi)}F(Z,\theta_1,\theta_2,\theta_3,\theta_4,\theta_5,\theta_6)+ \nu(\theta_1,\theta_2,\theta_3)+
  \nu(\theta_1,\theta_5,\theta_6)\leqslant v_8.$$
\end{cor}
\begin{proof}
Follows  immediately by using Lemmas \ref{stima1} and \ref{stima3} and taking averages.
\end{proof}

Consider now an admissible $6$-tuple $(\theta_1,\theta_2,\theta_3,\theta_4,\theta_5,\theta_6).$ Using Lemma \ref{sym}, we have that either all the variables  $\theta_i$ are greater than $\pi$ or at most one of them is greater than $\pi.$
The third technical lemma we need for the proof of Theorem \ref{prop:bound} is the following.    It  implies that the case where all the $\theta_i$'s are larger than $\pi$ can also be reduced to the case where all of them  are less than or equal to $\pi.$

 \begin{lem}\label{symm}  If $\theta_i>\pi$ for $i=1,\dots,6$ and $\alpha_i=\theta_i-\pi,$ then
  $$V\left(\theta_1,\dots,\theta_6\right)=V\left(\alpha_1,\dots,\alpha_6\right).$$
 \end{lem}

Now, assuming Lemma \ref{symm}, we conclude the proof of Theorem \ref{prop:bound}.
In the case where at most one $\theta_i>\pi,$ we can assume by symmetry that $\theta_i\leqslant \pi$ for all $i>1.$ Then Lemmas \ref{stima1}, \ref{stima3} and Corollary \ref{stimacor} imply that
\begin{align*}
 &V(\theta_1,\theta_2,\theta_3,\theta_4,\theta_5,\theta_6)\leqslant \\ &\max_{\max (U_i)\leqslant Z \leqslant \min (V_j,2\pi)}F(Z,\theta_1,\theta_2,\theta_3,\theta_4,\theta_5,\theta_6)+ \nu(\theta_1,\theta_2,\theta_3)+
  \nu(\theta_1,\theta_5,\theta_6)\leqslant v_8.
\end{align*}

In conclusion, we obtain
$$ \frac{2 \pi}{r}\log\abs{\Bigg|\begin{matrix}
   n_1 &n_2&n_3\\
   n_4&n_5&n_6
  \end{matrix}\Bigg|_{q=e^{\frac{2\pi i}{r}}}}\leqslant v_8+O\left(\frac{\log(r)}{r}\right).$$
  \end{proof}
  
 \begin{oss}
 In \cite{DK} a less sharp upper bound on the growth rate of the quantum $6j$-symbol was given, to prove that if a compact 3-manifold $M$ admits a triangulation with $t$ tetrahedra, then
 \begin{displaymath}
 lTV(M)\leqslant LTV(M) \leqslant 2.08v_8t.
 \end{displaymath}
The improvement of the upper bound allows us to state a better estimate. We have the following.
\begin{cor}\label{tetbound}
 If $M$ is a compact manifold that admits a triangulation with $t$ tetrahedra, then
 \begin{displaymath} 
 lTV(M)\leqslant LTV(M) \leqslant v_8t.
 \end{displaymath}
\end{cor}
\end{oss}
\begin{oss}
There is a concept of complexity of a manifold that is related to quantum invariants, the so called shadow complexity.
 For an overview of shadows and shadow complexity, see for example \cite[Part 2]{Turaevbook} or \cite{CosThurston}. Shadow complexity easily gives a bound on the growth of the Turaev-Viro invariants.
 \begin{cor}
  If $M$ has shadow complexity $c,$ then
   \begin{displaymath}
lTV(M)\leqslant  LTV(M)\leqslant 2 c v_8.
   \end{displaymath}
 Furthermore we have equalities for fundamental shadow links.
 \end{cor}
\begin{proof}
 The inequality is an immediate consequence of Theorem \ref{prop:bound} and the shadow formula for the Reshetikhin-Turaev invariants \cite[Theorem X.3.3]{Turaevbook}.
 By \cite{CosThurston}, for  fundamental shadow link in  $\#^{c+1}(S^2\times S^1)$ the shadow complexity is $c.$
 Hence sharpness follows from Theorem \ref{main}, which we prove in  Section \ref{secvol}.
\end{proof}

Moreover, shadow complexity also gives an upper bound on the simplicial volume.
\begin{teo*}[\cite{CosThurston}, Theorem 3.37]
 Let $M$ be a manifold with (possibly empty) toroidal boundary, simplicial volume $\textrm{Vol}(M)$ and shadow complexity $c;$ then, $\textrm{Vol}(M)\leqslant 2cv_8.$
 Furthermore this bound is sharp for complements of fundamental shadow links.
\end{teo*}
\end{oss}
\begin{oss} The bound in Corollary \ref{tetbound} is likely not sharp. However in \cite{DK} it is used to show that for 3-manifolds $M$ with toroidal or empty boundary
$LTV(M)$ is bounded above linearly by the Gromov norm of $M.$ On the other hand, the Gromov norm upper bound of the shadow complexity obtained in \cite{CosThurston} is quadratic.
\end{oss}

Before we move on to prove the volume conjecture for fundamental shadow links, we need to show that the bound of Theorem \ref{prop:bound} is sharp.

 \begin{lem}\label{r/2}
 If the sign is chosen such that $\frac{r\pm1}{2}$ is even, then
$$ \lim_{r\ra\infty}\frac{2 \pi}{r}\log\abs{\Bigg|\begin{matrix}
   \frac{r\pm1}{2} &\frac{r\pm1}{2}&\frac{r\pm1}{2}\\
   \frac{r\pm1}{2}&\frac{r\pm1}{2}&\frac{r\pm1}{2}
  \end{matrix}\Bigg|_{q=e^{\frac{2\pi i}{r}}}}=v_8.$$

 \end{lem}
 \begin{proof}
  
 Because of the color choice, then $\max T_i> \frac{r}{2},$ hence in the sum defining the quantum $6j$-symbol $\frac{r}{2}< z< r,$ and $\{z\}=2i\sin\left(2\pi  z/r\right)$ is an imaginary number with negative sign. Moreover, $0\leqslant z-T_i<\frac{r}{2}$ and $0\leqslant Q_j-z<\frac{r}{2}$ for all $i,j.$ Therefore,
 \begin{displaymath}
  \frac{(-1)^z\{z+1\}!}{\prod_{i=1}^4\{z-T_i\}!\prod_{j=1}^3\{Q_j-z\}!}
 \end{displaymath}
is an imaginary number, and passing from $z$ to $z+1$ in the sum does not change its sign, since all terms in the denominator do not change sign, and there is a change of sign due to $\{z+2\}$ that gets corrected by $(-1)^{z+1}.$ Since there is no change in sign among the summands, the estimate given by \eqref{estimate} is actually an equality. We have $\Delta\left(\pi,\pi,\pi\right)=0,$ and $$F\left(\frac{7\pi}{4},\pi,\pi,\pi,\pi,\pi,\pi\right)=8\Lambda\left(\frac{\pi}{4}\right)=v_8.$$
Thus, using Theorem \ref{prop:bound},
$$ \frac{2 \pi}{r}\log\abs{\Bigg|\begin{matrix}
   \frac{r\pm1}{2} &\frac{r\pm1}{2}&\frac{r\pm1}{2}\\
   \frac{r\pm1}{2}&\frac{r\pm1}{2}&\frac{r\pm1}{2}
  \end{matrix}\Bigg|_{q=e^{\frac{2\pi i}{r}}}}= v_8+O\left(\frac{\log(r)}{r}\right)$$
  which concludes the proof.
  \end{proof}


\section{Proofs of the technical lemmas}\label{sectec}

We now turn to the proofs of Lemmas \ref{stima1}, \ref{stima3} and \ref{symm}.

\begin{lem:stima1}
 If 
 $0\leqslant\alpha,\beta,\gamma\leqslant \pi,$
then $\nu(\alpha,\beta,\gamma)\leqslant 0.$
\end{lem:stima1}
\begin{proof}
 Put $x=\frac{\alpha+\beta-\gamma}{2},$ $y=\frac{\alpha-\beta+\gamma}{2},$ $z=\frac{-\alpha+\beta+\gamma}{2}.$
 Then we need to maximize 
 \begin{displaymath}
 \nu(\alpha,\beta,\gamma)=\vartheta(x,y,z)=\frac{1}{2}(\Lambda(x+y+z)-\Lambda(x)-\Lambda(y)-\Lambda(z))
 \end{displaymath}
 with the constraints 
$0\leqslant x+y\leqslant \pi,$ $0\leqslant x+z\leqslant\pi$ and $0\leqslant y+z\leqslant\pi.$

  To do this, we check first its stationary points in the interior of the domain, then we explore the boundary, and finally the points where $\vartheta$ is not smooth.
 \begin{gather}
  \frac{\partial \vartheta(x,y,z)}{\partial x}=- \frac{1}{2}(\log(2|\sin(x+y+z)|)-\log(2|\sin(x)|));\label{dertheta}\\
  \frac{\partial \vartheta(x,y,z)}{\partial y}=- \frac{1}{2}(\log(2|\sin(x+y+z)|)-\log(2|\sin(y)|));\\
  \frac{\partial \vartheta(x,y,z)}{\partial z}= -\frac{1}{2}(\log(2|\sin(x+y+z)|)-\log(2|\sin(z)|))
   \end{gather}
 So by putting them all equal to $0,$ we first see that $\sin(x)=\pm \sin(y)=\pm \sin(z),$ so either $x=y=z$ modulo $\pi$ or one of $x+y,$ $y+z$ or $x+z$ is equal $k\pi$ for some $k\in\Z.$ Suppose $x+y=k\pi.$ Then
 \begin{displaymath}
  \vartheta(x,y,z)=\Lambda(k\pi+z)-\Lambda(k\pi-y)-\Lambda(y)-\Lambda(z)=0;
 \end{displaymath}
because $\Lambda$ is odd and $\pi$-periodic; $y+z=k\pi$ and $x+z=k\pi$ are the same by symmetry.

 If instead $x=y=z$ modulo $\pi,$ substituting $x=y=z$ in \eqref{dertheta}, we get $\sin(3x)=\pm \sin(x).$ This means that $x=y=z=\frac{k\pi}{4}$ modulo $\pi.$ In the interior of the domain
 this implies $x=y=z=\frac{\pi}{4};$ all other possibilities lie outside the domain or on its boundary.
In this point $\vartheta=-2\Lambda\left(\frac{\pi}{4}\right)\cong -1.83<0.$
 
 The boundary cases $x+y=k\pi$ and permutations were already checked, finding $\vartheta=0.$ 
 
 Finally we check the points where $\vartheta$ is not smooth. This happens when one of the following holds:
 \begin{itemize}
  \item $x=k\pi,$ or $y=k\pi,$ or $z=k\pi;$ or
  \item $x+y+z=k\pi.$
 \end{itemize}
 
 \begin{oss}
  If $P$ is a point and $\gamma$ is a direction such that the derivative of $\vartheta$ in that direction is $+\infty,$ then $P$ cannot be a local maximum of $\vartheta.$
 \end{oss}
 If $x=k\pi,$ then $\frac{\partial \vartheta(x,y,z)}{\partial x}=+\infty$ unless $x+y+z=h\pi,$ and $(x,y,z)$ cannot be a maximum. If instead $x=k\pi$ and $x+y+z=h\pi,$ we have $y+z=(h-k)\pi$ and we are in a case
 we already checked. $y=k\pi$ and $z=k\pi$ are symmetric.
 
 If instead $x+y+z=k\pi,$ we find once again an infinite derivative unless $x=h\pi,$ and we reason as before.
 So in conclusion $\vartheta$ is equal to $0$ on the boundary of the set $\{0\leqslant x+y\leqslant \pi,0\leqslant x+z\leqslant\pi,0\leqslant y+z\leqslant\pi\},$ cannot have a maximum in a non-smooth point
 and has a unique stationary point in the interior, where it is negative. This concludes the proof.
\end{proof}

We will need the following lemma.

 \begin{lem}\label{conv}
  If $0\leqslant a,b$ and $a+b\leqslant 2\pi,$ then
  \begin{displaymath}
   -v_3\leqslant \Lambda(a+b)-\Lambda(a)-\Lambda(b)\leqslant v_3
  \end{displaymath}
where $v_3=\Lambda\left(\frac{\pi}{3}\right)\cong 1.01$ is the volume of the regular ideal tetrahedron.
 \end{lem}
 \begin{proof}
 First notice that if $a+b=k\pi,$ then because $\Lambda$ is odd and $\pi$-periodic, we have $\Gamma(a,b)=\Lambda(a+b)-\Lambda(a)-\Lambda(b)=0.$ Similarly if $a=0$ or $b=0$ then $\Gamma(a,b)=0.$ 
 By calculating the derivatives of $\Gamma$ and putting them to $0$ we obtain, reasoning as before, $a=\pm b$ modulo $\pi.$ If $a=-b$ modulo $\pi$ then we have seen that $\Gamma=0.$ Then $a=b$ implies $\sin(2a)=\pm \sin(a),$ and
 either $a=k\pi$ (in which case $\Gamma=0$) or $3a=k\pi.$ If $a=\frac{\pi}{3}$ we obtain $\Gamma\left(\frac{\pi}{3},\frac{\pi}{3}\right)=-3\Lambda(\frac{\pi}{3})=-v_3,$ while $a=\frac{2\pi}{3}$ implies 
 $\Gamma\left(\frac{2\pi}{3},\frac{2\pi}{3}\right)=3\Lambda(\frac{\pi}{3})=v_3.$
 \end{proof}
 
 We are now ready to prove Lemma \ref{stima3}.
 \begin{lem:stima3}
 If $0\leqslant\theta_1,\theta_2,\theta_3,\theta_4,\theta_5,\theta_6\leqslant 2\pi$ and $max(T_i)\leqslant Z\leqslant min(Q_j,2\pi),$ then 
 \begin{displaymath}
  F(Z,\theta_1,\theta_2,\theta_3,\theta_4,\theta_5,\theta_6)+2\nu(\theta_1,\theta_2,\theta_3)\leqslant 8\Lambda\left(\frac{\pi}{4}\right)=v_8
 \end{displaymath}

\end{lem:stima3}
 
 \begin{proof}
 Put $a_i=Z-U_i,$ and $b_j=V_j-Z.$ The inverse of this change of variable is as follows.
 \begin{itemize}
 \item $\theta_1=a_3+a_4+b_1+b_2;$
 \item $\theta_2=a_2+a_4+b_1+b_3;$
 \item $\theta_3=a_2+a_3+b_2+b_3;$
  \item $\theta_4=a_1+a_2+b_1+b_2;$
  \item $\theta_5=a_1+a_3+b_1+b_3;$
  \item $\theta_6=a_1+a_4+b_2+b_3$ and
  \item $Z=a_1+a_2+a_3+a_4+b_1+b_2+b_3.$
 \end{itemize}

 In these new variables we have,
 \begin{gather*}
  F(Z,\theta_1,\theta_2,\theta_3,\theta_4,\theta_5,\theta_6)=\tilde{F}(a_1,a_2,a_3,a_4,b_1,b_2,b_3)=\\-\Lambda\left(\sum_{i=1}^4 a_i+\sum_{j=1}^3 b_j\right)+\sum_{i=1}^4 \Lambda(a_i)+\sum_{j=1}^3 \Lambda(b_j)
 \end{gather*}

 while
 \begin{gather*}
  2\nu(\theta_1,\theta_2,\theta_3)=2\tilde{\nu}(a_1,a_2,a_3,a_4,b_1,b_2,b_3)=\\\left(\Lambda\left(\sum_{i=1}^3(a_i+b_i)\right)-\sum_{i=1}^3\Lambda(a_i+b_i)\right).
 \end{gather*}

Let $L=2\tilde{\nu}+\tilde{F},$ and notice that $\tilde\nu$ is independent of $a_4,$ and that $L$ is symmetric under the exchange of $a_i$ with $b_i$ for any $i\neq 4,$ and under 
 $$(a_1,a_2,a_3,a_4,b_1,b_2,b_3)\ra (a_{\sigma_1},a_{\sigma_2},a_{\sigma_3},a_4,b_{\sigma_1},b_{\sigma_2},b_{\sigma_3})$$ where $\sigma$ is any permutation of $3$ elements. Also notice that $L$ is periodic of period $\pi$ in each variable, hence we can assume $0\leqslant a_i\leqslant \pi$ and $0\leqslant b_i\leqslant \pi.$
 Moreover, because of the constraints on the $\theta$s and on $Z,$ we have that $0\leqslant \sum a_i+\sum b_j\leqslant 2\pi.$
 Denote with $\Omega$ the region of $\R^7$ defined by all these inequalities.
 
 We now proceed by first dealing with the points in the boundary of $\Omega,$ then
 with the points where the function $L$ is not differentiable, and finally by finding the stationary points in the interior of $\Omega.$
 Start by calculating the partial derivatives of $L.$
 \begin{gather}
 \label{dera4} \frac{\partial L}{\partial a_4}=\log\left|\frac{\sin(a_1+a_2+a_3+a_4+b_1+b_2+b_3)}{\sin(a_4)}\right|\\\label{dera1}
  \frac{\partial L}{\partial a_1}=\log\left|\frac{\sin(a_1+a_2+a_3+a_4+b_1+b_2+b_3)\sin(a_1+b_1)}{\sin(a_1)\sin(a_1+a_2+a_3+b_1+b_2+b_3)}\right|\\ \label{dera2}
  \frac{\partial L}{\partial a_2}=\log\left|\frac{\sin(a_1+a_2+a_3+a_4+b_1+b_2+b_3)\sin(a_2+b_2)}{\sin(a_2)\sin(a_1+a_2+a_3+b_1+b_2+b_3)}\right|\\ \label{dera3}
  \frac{\partial L}{\partial a_3}=\log\left|\frac{\sin(a_1+a_2+a_3+a_4+b_1+b_2+b_3)\sin(a_3+b_3)}{\sin(a_3)\sin(a_1+a_2+a_3+b_1+b_2+b_3)}\right|\\ \label{derb1}
  \frac{\partial L}{\partial b_1}=\log\left|\frac{\sin(a_1+a_2+a_3+a_4+b_1+b_2+b_3)\sin(a_1+b_1)}{\sin(b_1)\sin(a_1+a_2+a_3+b_1+b_2+b_3)}\right|\\ \label{derb2}
  \frac{\partial L}{\partial b_2}=\log\left|\frac{\sin(a_1+a_2+a_3+a_4+b_1+b_2+b_3)\sin(a_2+b_2)}{\sin(b_2)\sin(a_1+a_2+a_3+b_1+b_2+b_3)}\right|\\ 
  \frac{\partial L}{\partial b_3}=\log\left|\frac{\sin(a_1+a_2+a_3+a_4+b_1+b_2+b_3)\sin(a_3+b_3)}{\sin(b_3)\sin(a_1+a_2+a_3+b_1+b_2+b_3)}\right| \label{derb3}.
 \end{gather}
 
The remaining of the proof is broken into thee steps.

 \emph{Step 1: the boundary points}

 Suppose we have a maximum for $L$ in a point $P$ in the boundary of $\Omega.$ If $a_1=\pi,$ then by periodicity we would have a maximum with $a_1=0,$ so we study this case instead. The derivative of $L$ \eqref{dera1} with respect to $a_1$ is $+\infty$ if $a_1+b_1\neq k\pi$ and 
 $a_1+a_2+a_3+a_4+b_1+b_2+b_3\neq k\pi,$ and
 we would not get a maximum. Hence, either $a_1+a_2+a_3+a_4+b_1+b_2+b_3=k\pi$ or $b_1=k\pi.$ In the first case, we have that
 \begin{displaymath}
  L=\Lambda(a_2)+\Lambda(b_2)-\Lambda(a_2+b_2)+\Lambda(a_3)+\Lambda(b_3)-\Lambda(a_3+b_3)
 \end{displaymath}
and using Lemma \ref{conv} we find $L\leqslant2v_3.$ In the second case,
\begin{align*}
 L=&\Lambda(a_2)+\Lambda(b_2)-\Lambda(a_2+b_2)+\Lambda(a_3)+\Lambda(b_3)-\Lambda(a_3+b_3)\\&+\Lambda(b_2+b_3+a_2+a_3)+\Lambda(a_4)-\Lambda(b_2+b_3+a_4+a_2+a_3)
\end{align*}
and again Lemma \ref{conv} implies $L\leqslant 3v_3.$ 
If $a_4=0,$ the same reasoning implies that $P$ cannot be a maximum unless $a_1+a_2+a_3+a_4+b_1+b_2+b_3=k\pi,$ and in this case
\begin{align*}
 L=&\Lambda(a_1)+\Lambda(b_1)-\Lambda(a_1+b_1)+\Lambda(a_2)+\Lambda(b_2)+\\&-\Lambda(a_2+b_2)+ \Lambda(a_3)+\Lambda(b_3)-\Lambda(a_3+b_3)\leqslant 3v_3.
\end{align*}
If $a_1+a_2+a_3+a_4+b_1+b_2+b_3=k\pi$ once again we would have $\frac{\partial L}{\partial (-a_4)}=+\infty$ unless $a_4=0$ and we would be in the same case as before. The remaining cases are dealt by symmetry.

  \emph{Step 2: the non-smooth points}
  
 First off, notice that $L$ is differentiable at $P=(a_1,a_2,a_3,a_4,b_1,b_2,b_3)$ unless one (or more) of the following equalities (considered modulo $\pi$) holds:
 \begin{enumerate}
  \item $a_i=0$ for some $i;$
  \item $b_j=0$ for some $j;$
  \item $a_i+b_i=0$ for some $i;$
  \item $a_1+a_2+a_3+a_4+b_1+b_2+b_3=0;$
  \item $a_1+a_2+a_3+b_1+b_2+b_3=0.$
 \end{enumerate}
 These cases are dealt in a similar fashion as the boundary cases. 
 
 Suppose we have a maximum for $L$ in a point $P$ such that $a_1+a_2+a_3+b_1+b_2+b_3=k\pi.$ Then, unless $a_1+b_1=k\pi$ or $a_1+a_2+a_3+a_4+b_1+b_2+b_3=k\pi$ the derivative of $L$ with respect to $a_1$ is $+\infty,$ hence $P$
 could not be a maximum. Using Lemma \ref{conv} we obtain that in the first case,
 \begin{equation*}
   L=\Lambda(a_2)+\Lambda(b_2)-\Lambda(a_2+b_2)+\Lambda(a_3)+\Lambda(b_3)-\Lambda(a_3+b_3)\leqslant 2v_3
 \end{equation*}
and in the second
\begin{align*}
 L=&\Lambda(a_2)+\Lambda(b_2)-\Lambda(a_2+b_2)+\Lambda(a_3)+\Lambda(b_3)-\Lambda(a_3+b_3)+\\
 &\Lambda(b_2+b_3+a_2+a_3)+\Lambda(a_4)-\Lambda(b_2+b_3+a_4+a_2+a_3)\leqslant 3v_3.
\end{align*}
 The cases $a_i=k\pi,$ $b_j=k\pi,$ or $a_1+a_2+a_3+a_4+b_1+b_2+b_3=k\pi$ were already addressed before.
 If $a_1+b_1=0,$ then $a_1=b_1=0$ and it was already addressed. If $a_1+b_1=k\pi>0,$ then the derivative of $L$ in the direction $-a_1$ is $+\infty$ unless $a_1=0$ or  $a_1+a_2+a_3+b_1+b_2+b_3=k\pi,$ which are both cases
 we have dealt with already. The remaining cases are done by the symmetries of $L.$
 
 \emph{Step 3: the interior smooth points}

Now we turn to the smooth points in the interior of $\Omega.$
By equating \eqref{dera1} and \eqref{derb1} to $0,$ we find $\sin(a_1)=\pm \sin(b_1).$ Similarly $\sin(a_i)=\pm \sin(b_i)$ for $i=2,3$ by equating \eqref{dera2} to \eqref{derb2} and \eqref{dera3} to \eqref{derb3} respectively.
Because of the boundary and smoothness 
conditions, we have that in the interior of the domain this implies 
$a_i=b_i$ for $i=1,2,3.$ By putting equations \eqref{dera1} and \eqref{dera2} to $0,$ we find
\begin{equation*}
 \frac{\sin(2a_1)}{\sin{a_1}}=\pm\frac{\sin(2a_2)}{\sin{a_2}}.
\end{equation*}
Which implies that $\cos(a_1)=\pm\cos(a_2)$ and either $a_1=a_2$ or $a_1+a_2=\pi.$ However, if $a_1+a_2=\pi,$ we would have $a_1+a_2+a_3+a_4+b_1+b_2+b_3=2a_1+2a_2+2a_3+a_4\geqslant 2 \pi;$ hence, this is not possible in the interior of $\Omega.$ Similarly $a_1=a_3.$

Now by putting equation \eqref{dera4} equal to $0$ we obtain
\begin{equation*}
 \sin\left(6a_1+a_4\right)=\pm \sin(a_4)
\end{equation*}
This implies either $6a_1=k\pi$ or $6a_1+2a_4=k\pi,$ but in the first case we would not be in a smooth point (case $5$ of the previous step). By plugging everything we obtained in equation \eqref{dera1} we finally find
\begin{equation*}
 \frac{\sin(a_4)\sin(2a_1)}{\sin(a_1)\sin(2a_4)}=\pm 1
\end{equation*}
Hence $a_4=a_1$ or $a_4=\pi-a_1.$ Both cases imply that the stationary points of $L$ must be of the form $\left(\frac{k\pi}{8},\frac{k\pi}{8},\frac{k\pi}{8},\frac{k\pi}{8},\frac{k\pi}{8},\frac{k\pi}{8},\frac{k\pi}{8}\right),$ 
for $k=1,2.$ In the first case $L\cong 3.01<v_8,$ while in the second $L=8\Lambda\left(\frac{\pi}{4}\right)=v_8.$
\end{proof}

We conclude the section with the proof of Lemma \ref{symm}.

 \begin{lem:symm}
 If $\theta_i>\pi$ for $i=1,\dots,6$ and $\alpha_i=\theta_i-\pi,$ then
  $$V\left(\theta_1,\dots,\theta_6\right)=V\left(\alpha_1,\dots,\alpha_6\right).$$
 \end{lem:symm}

\begin{proof}
The value of $V\left(\alpha_1,\dots,\alpha_6\right)$ is equal, by the Murakami-Yano-Ushijima formula \cite[Theorems 1 and 2]{MY}, \cite[Theorem 1.1]{vol}, to the volume of the hyperbolic truncated tetrahedron with external dihedral angles $\alpha_1,\dots,\alpha_6.$
Thus we need to show that this formula is symmetric under the change $\left(\theta_1,\dots,\theta_6\right)\leftrightarrow\left(\alpha_1,\dots,\alpha_6\right).$ We now pass to the internal
dihedral angles $(\xi_1,\dots,\xi_6)$ with $\xi_i=\pi-\alpha_i,$ as these are more natural for the Murakami-Yano-Ushijima formula. In these variables, the formula reads
\begin{equation*}
 V(a_1,\dots,a_6):=\frac{1}{2}\textrm{Im}\left(U(z_1,\vec{a})-U(z_2,\vec{a})\right)
\end{equation*}
where we have
\begin{itemize}
\item $a_i=e^{\sqrt{-1}\xi_i};$
 \item \begin{align}\begin{split}
        U(z,\vec{a})=&\frac{1}{2}(\li(z)+\li(za_1a_2a_4a_5)+\li(za_1a_3a_4a_6)+\\&
 +\li(za_2a_3a_5a_6)-\li(-za_1a_2a_3)-\li(-za_1a_5a_6)+\\&-\li(-za_2a_4a_6)-\li(-za_3a_4a_5)),\end{split}
       \end{align}
where $\li$ is the dilogarithm function defined for $z\in \mathbb C\setminus [1,\infty)$ by 
$$\li(z)=-\int_0^z\frac{\log(1-u)}{u}du;$$

 \item $z_1$ and $z_2$ are the solutions of the equation $\alpha+\beta z+\gamma z^2=0,$ labeled in such a way as to obtain a positive value for $V;$
\item \begin{align}\begin{split}\label{alpha}\alpha=&1+a_1a_2a_4a_5+a_1a_3a_4a_6+a_2a_3a_5a_6+a_1a_2a_3+a_1a_5a_6+\\&a_2a_4a_6+a_3a_4a_5;\end{split}\end{align}
\item \begin{align}\begin{split}\label{beta}
       \beta=&-a_1a_2a_3a_4a_5a_6\big((a_1-a_1^{-1})(a_4-a_4^{-1})+\\&+(a_2-a_2^{-1})(a_5-a_5^{-1})+(a_3+a_3^{-1})(a_6-a_6^{-1})\big);\end{split}
      \end{align}

\item \begin{align}\begin{split}\label{gamma}\gamma=&a_1a_2a_3a_4a_5a_6(a_1a_2a_3a_4a_5a_6+a_1a_4+a_2a_5+a_3a_6+\\&+a_1a_2a_6+a_1a_3a_5+a_2a_3a_4+a_4a_5a_6).\end{split}\end{align}
 \end{itemize}
 
 In these variables, the symmetry we need to explore is 
 $$(a_1,a_2,a_3,a_4,a_5,a_6)\leftrightarrow(a_1^{-1},a_2^{-1},a_3^{-1},a_4^{-1},a_5^{-1},a_6^{-1}).$$

Call $\alpha,$ $\beta,$ and $\gamma$ as in formulas \eqref{alpha}-\eqref{gamma}, and $\alpha',\beta',\gamma'$ the same formulas with $a_i\ra a_i^{-1}$ for all $i.$ 
Let $z_1$ and $z_2$ be solutions of $\alpha+\beta z+\gamma z^2=0.$ Now it is immediate to check that $$\alpha'=\frac{\gamma}{a_1^2a_2^2a_3^2a_4^2a_5^2a_6^2}, \ \  \beta'=\frac{\beta}{a_1^2a_2^2a_3^2a_4^2a_5^2a_6^2}, \ \ {\rm and} \ \ \gamma'=\frac{\alpha}{a_1^2a_2^2a_3^2a_4^2a_5^2a_6^2}.$$ Hence, we need to solve the equation
\begin{displaymath}
 \alpha'+\beta' z + \gamma'z^2=\frac{1}{a_1^2a_2^2a_3^2a_4^2a_5^2a_6^2}\left(\gamma+\beta z+\alpha z^2\right)=0;
\end{displaymath}
call the solutions $\hat{z}_1$ and $\hat{z}_2.$ Since it was shown in \cite[Page 384]{MY} that $z_1$ and $z_2$ must be complex numbers with absolute value $1,$ $\hat{z}_1=\overline{z_1}$ and $\hat{z}_2=\overline{z_2}.$
Now we can compute

\begin{align*}
U&(\hat{z}_1,a_1^{-1},a_2^{-1},a_3^{-1},a_4^{-1},a_5^{-1},a_6^{-1})=\frac{1}{2}(\li(\overline{z_1})+\li(\overline{z_1a_1a_2a_4a_5})+\\&+\li(\overline{z_1a_1a_3a_4a_6})
 +\li(\overline{z_1a_2a_3a_5a_6})-\li(-\overline{z_1a_1a_2a_3})-\li(-\overline{z_1a_1a_5a_6})+\\&-\li(-\overline{z_1a_2a_4a_6})-\li(-\overline{z_1a_3a_4a_5})).
 \end{align*}
 
Because $\li(\overline{a})=\overline{\li(a)},$ we see that $$U(\hat{z}_1,a_1^{-1},a_2^{-1},a_3^{-1},a_4^{-1},a_5^{-1},a_6^{-1})=\overline{U(z_1,a_1,a_2,a_3,a_4,a_5,a_6)};$$ and
$$U(\hat{z}_2,a_1^{-1},a_2,a_3,a_4^{-1},a_5,a_6)=\overline{U(z_2,a_1,a_2,a_3,a_4,a_5,a_6)}.$$ Because we have to switch the labels as to obtain a positive value of $V,$ we finally obtain
$V(a_1^{-1},a_2^{-1},a_3^{-1},a_4^{-1},a_5^{-1},a_6^{-1})=V(a_1,a_2,a_3,a_4,a_5,a_6).$
\end{proof}


 \section{The volume conjecture for fundamental shadow links}\label{secvol}

In this section we define the family of fundamental shadow links and prove the volume conjecture for them.
The building block for these links is a $3$-ball with $4$ disks on its boundary, and $6$ arcs connecting them, as in picture \ref{fig:bblock}. If we take $c$ building blocks $B_1,\dots,B_c$ and glue them together along the disks, in such a way that
each endpoint of each arc is glued to some other endpoint (possibly of the same arc), we obtain a (possibly non-orientable) handlebody of genus $c+1$ with a link in its boundary, such as in picture \ref{fig:linkhand}. By taking the orientable double of this handlebody (the orientable double cover whose boundary is quotiented by the deck involution), we obtain a link inside
$M_c:=\#^{c+1}(S^1\times S^2).$ We call a link obtained this way a {fundamental shadow link}.

The most important features of these links are that their geometry and quantum invariants are well understood.
\begin{lem}\cite[Proposition 3.33]{CosThurston}
If $L\subseteq M_c$ is a fundamental shadow link, then $M_c\setminus L$ is hyperbolic of volume $2cv_8$ and shadow complexity $c.$
\end{lem}

The next lemma follows from the shadow reformulation of the $SO(3)$ version of the  Reshetikhin-Turaev  invariants \cite{TuJDG, Turaevbook}. The proof is given in
\cite[Proposition 4.1]{C} following 
 \cite{C:volconj}.

\begin{lem}\label{shform}
  Let $J_r=\{0,2,\dots,r-3\}.$ If $L=L_1\sqcup\cdots\sqcup L_k\subseteq M_c$ is a fundamental shadow link and $col\in J_r^k$ is a coloring of its components with even numbers, then
 \begin{equation*}
  RT_r\left(M_c,L,col\right)= \left(\frac{2\sin\left(\frac{2\pi}{r}\right)}{\sqrt{r}}\right)^{-c}\prod_{i=1}^c\begin{vmatrix}
   col(i_1) &col(i_2)&col(i_3)\\
   col(i_4)&col(i_5)&col(i_6)
  \end{vmatrix}
 \end{equation*}
where $i_j$ is the component of the link $L$ passing through the $j$-th strand of block $i.$
\end{lem}

Finally, we recall that any compact oriented $3$-manifold with toroidal or empty boundary is obtained as a Dehn filling of some of the boundary components of the complement of some fundamental shadow link \cite[Proposition 3.36]{CosThurston}.

 To relate the Turaev-Viro invariant of $M_c\setminus L$ to the Reshetikhin-Turaev invariant of $\left(M_c,L\right)$ of Lemma \ref{shform} we use the following proposition.
 It first appeared in \cite{DKY} in a slightly weaker version; we give essentially the same proof, slightly modified when needed.

\begin{prop}\label{tvtort}
 For any link $L=L_1\sqcup\cdots\sqcup L_k$ in a closed oriented $3$-manifold $M,$
 \begin{equation*}
  TV_r(M\setminus L)=2^{b_2\left(M\setminus L\right)}\sum_{col\in J_r^k}\left|RT_r\left(M,L,col\right)\right|^2,
 \end{equation*}
where $b_2\left(M\setminus L\right)$ denotes the rank of $H_2\left(M\setminus L, \ {\Z}_2\right).$

\end{prop}
\begin{proof}
 For a compact, oriented $3$-manifold $X$ with toroidal boundary let $DX$ denote the double of $X$ along $\partial X$ and let
  $b_2\left(X\right)$ denote the rank of $H_2\left(X, \ {\Z}_2\right).$ By \cite[Theorems 2.9 and 3.2]{BePe} for the case $q=e^{\frac{\pi i}{r}},$ adapted to other roots of unity in \cite[Theorems 2.9 and 3.1]{DKY},
we have
 \begin{displaymath}
  TV_r(X)=2^{b_2(M)}RT_r(DX).
 \end{displaymath}
Now let $X=M\setminus L.$ Because of the axioms of the TQFT associated to the Reshetikhin-Turaev invariants, we have $RT_r(DX)=\langle Z_r(X),Z_r(X)\rangle,$ where $Z_r(X)$ is
the vector in the $SO(3)$ Reshetikhinn-Turaev TQFT hermitian vector space $V_r(\partial X).$ 

The boundary of $X$ is a union of connected toroidal components $T_1\sqcup\cdots\sqcup T_k,$ and $V_r(\partial X)=V_r(T_1)\otimes\cdots\otimes V_r(T_k).$
An orthogonal basis for the vector space $V_r(T_i)$ is $(e_j)_{j\in J_r}$ where $e_j$ is the solid torus with boundary $T_i$ and whose core is colored with color $j.$ Therefore, an orthogonal basis for 
$V_r(\partial X)$ is $(e_{j_1}\otimes\cdots\otimes e_{j_k})_{j\in J_r^k}.$ Written in this basis, by the definition of the Reshetikhin-Turaev invariants,
\begin{displaymath}
 Z_r(X)=\sum_{col\in J_r^k} RT_r(M,L,col) e_{col_1}\otimes\cdots\otimes e_{col_k}
\end{displaymath}
hence
\begin{displaymath}
 \langle Z_r(X),Z_r(X)\rangle=\sum_{col\in J_r^k}\left|RT_r(M,L,col)\right|^2
\end{displaymath}

concluding the proposition.
\end{proof}

\begin{figure}
 \centering
  \begin{tikzpicture}[scale=0.8]
 \draw[thick][rotate=-45] (0,0) ellipse (1cm and 0.5cm);
 
 \draw[thick][rotate around={-45:(4,4)}] (4,4) ellipse (1cm and 0.5cm);

 \draw[thick][rotate around={45:(4,0)}] (4,0) ellipse (1cm and 0.5cm);
 \draw[thick][rotate around={45:(0,4)}] (0,4) ellipse (1cm and 0.5cm);
\draw [thick](4.71,0.71)[blue] to[out=135, in=-135] (4.71,3.29);

\draw [thick](0.71,-0.71)[red] to[out=45, in=135] (3.29,-0.71);
\draw [thick](-0.71,0.71)[green] to[out=45, in=-45] (-0.71,3.29);
\draw [thick](0.71,4.71)[orange] to[out=-45, in=-135] (3.29,4.71);
\draw [thick](0.36,0.36)[brown] -- (3.64,3.64);
\draw [thick](-0.36,4.36)[magenta,dashed]--(4.36,-0.36);
\end{tikzpicture}
\caption{The building block}\label{fig:bblock}
\end{figure}
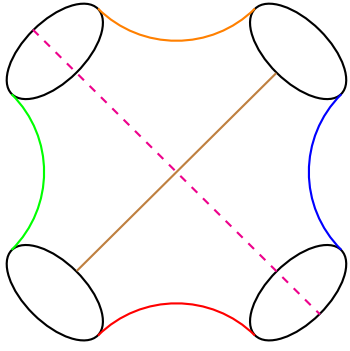

\begin{figure}
 \centering
 \begin{tikzpicture}[scale=0.7]
\draw [orange, thick] plot [smooth cycle] coordinates {(0,0) (2,2) (5,1) (8,2) (11,1) (14,2) (16,0) (14,-2) (11,-1) (8,-2) (5,-1) (2,-2)};
\draw [thick] plot [smooth cycle] coordinates {(-0.4,0) (2,2.4) (5,1.4) (8,2.4) (11,1.4) (14,2.4) (16.4,0) (14,-2.4) (11,-1.4) (8,-2.4) (5,-1.4) (2,-2.4)};
\draw [thick] plot [smooth] coordinates {(1.2,0.2) (1.6,-0.1) (2,-0.2) (2.4,-0.1) (2.8,0.2)};
\draw [thick] plot [smooth] coordinates {(7.2,0.2) (7.6,-0.1) (8,-0.2) (8.4,-0.1) (8.8,0.2)};
\draw [thick] plot [smooth] coordinates {(13.2,0.2) (13.6,-0.1) (14,-0.2) (14.4,-0.1) (14.8,0.2)};
\draw [thick] plot [smooth] coordinates {(1.45,0) (1.73,0.15) (2,0.2) (2.27,0.15) (2.55,0)};
\draw [thick] plot [smooth] coordinates {(7.45,0) (7.73,0.15) (8,0.2) (8.27,0.15) (8.55,0)};
\draw [thick] plot [smooth] coordinates {(13.45,0) (13.73,0.15) (14,0.2) (14.27,0.15) (14.55,0)};
\draw [green, thick] (2,0.15) ellipse (1.3 cm and 0.7 cm);
\draw [red, thick] (8,0.15) ellipse (1.3 cm and 0.7 cm);
\draw [blue, thick] (14,0.15) ellipse (1.3 cm and 0.7 cm);
\draw [magenta,dashed, thick] plot [smooth cycle] coordinates {(0.5,0) (2,1.5) (5,0.7) (8,1.5) (10.5,0) (8,-1.5) (5,-0.7) (2,-1.5) };
\draw [brown, thick] plot [smooth cycle] coordinates {(6.5,0) (8,1.5) (11,0.7) (14,1.5) (15.5,0) (14,-1.5) (11,-0.7) (8,-1.5) };
\end{tikzpicture}
\caption{The link in the boundary of the handlebody}\label{fig:linkhand}
\end{figure}
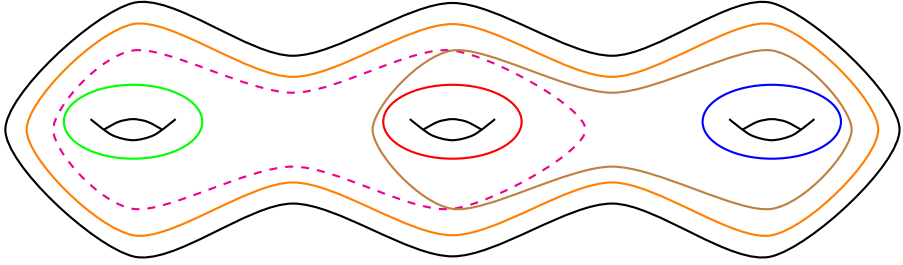

 We are ready to prove Conjecture \ref{volconj} for the complements of these links.
 \begin{teo:cong}
  For any fundamental shadow link $L=L_1\sqcup \dots \sqcup L_k$ built from $c$ blocks,
  \begin{equation*}
   \lim_{r\ra \infty} \frac{2 \pi}{r} \log\left| TV_r(M_c\setminus L)\right|=\textrm{Vol}(M_c\setminus L)=2c v_8.
  \end{equation*}
 \end{teo:cong}
\begin{proof}
 If $L=L_1\sqcup \dots \sqcup L_k$ we have by Proposition \ref{tvtort},
 $$
  TV_r(M_c\setminus L)=2^{b_2\left(M_c\setminus L\right)}\sum_{col\in J_r^k} \left|RT_r\left(M_c,L,col\right)\right|^2.
 $$

Because the number of possible colorings is polynomial in $r,$
\begin{displaymath}
 \frac{2\pi}{r} \log\left|TV_r\left(M_c\setminus L\right)\right|\leqslant \max_{col\in J_r^k} \frac{2\pi}{r}\log\left(\left| RT_r\left(M_c,L,col\right)\right|\right)^2+O\left(\frac{\log(r)}{r}\right).
\end{displaymath}

By Lemma \ref{shform}, we have that $RT_r(M_c,L,col),$ up to a factor that grows  polynomially in $r,$ is equal to
\begin{displaymath}
 \prod_{i=1}^c\begin{vmatrix}
   col(i_1) &col(i_2)&col(i_3)\\
   col(i_4)&col(i_5)&col(i_6)
  \end{vmatrix}
\end{displaymath}
where $i_j$ is the component of the link $L$ passing through the $j$-th strand of block $i.$ Hence, because of Theorem \ref{prop:bound},

\begin{displaymath}
 \lim_{r\ra\infty}\frac{2\pi}{r} \log\left|TV_r\left(M_c\setminus L\right)\right|
 \leqslant 2cv_8
\end{displaymath}

On the other hand, if we take $col=\left(\frac{r\pm1}{2},\dots,\frac{r\pm1}{2}\right)$ to be even colors, we have $$\lim_{r\ra \infty} \frac{2 \pi}{r} \log\left|TV_r\left(M_c\setminus L\right)\right|\geqslant  \lim_{r\ra\infty}\frac{2 \pi}{r}\log\abs{\Bigg|\begin{matrix}
   \frac{r\pm1}{2} &\frac{r\pm1}{2}&\frac{r\pm1}{2}\\
   \frac{r\pm1}{2}&\frac{r\pm1}{2}&\frac{r\pm1}{2}
  \end{matrix}\Bigg|_{q=e^{\frac{2\pi i}{r}}}}^{2c}=2cv_8$$
  by Lemma \ref{r/2}.
\end{proof}


\section{Applications} \label{FSL}


 In the previous sections we showed that the  Turaev-Viro invariants volume conjecture is true for the fundamental shadow links. As recalled in the Introduction, those links are universal in the sense  every orientable compact 3-manifold is obtained by a Dehn surgery along
 these links. On the other hand, the behavior of Turaev-Viro invariants under Dehn filling was studied in \cite{DK}. Here we combine these results with results about estimates of hyperbolic volume change under Dehn filling
 to derive some interesting applications.

\subsection{Dehn filling, volume  and Turaev-Viro invariants} 

Let  $N$ be a compact 3-manifold with toroidal boundary whose interior is hyperbolic, and let  $T_1, \ldots, T_k$ be some components of $\partial N.$
On each $T_i,$ choose a slope $s_i,$ such that the shortest length of any of the $s_i$ is denoted $\lmin.$  If $\lmin > 2\pi.$ then by the Geometrization Theorem, the manifold $M=N(s_1, \dots, s_k)$ obtained by Dehn filling along $s_1, \dots, s_k$ is hyperbolic. Moreover,  there is  a correlation between its volume and the volume of $N.$
  In the Theorem below, the upper inequality is by Thurston \cite[Theorem 6.5.6]{thurston:notes} and the lower inequality is 
  by the following  result proved in \cite{fkp:filling}.

\begin{teo}\cite[Theorem 1.1]{fkp:filling} \label{Thm:VolChange} 
Let  $N$ be a cusped hyperbolic $3$--manifold, containing embedded horocusps $C_1, \ldots, C_k$ (plus possibly others). On each torus $T_i = \bdy C_i,$ choose a slope $s_i,$ such that the shortest length of any of the $s_i$ is $\lmin > 2\pi.$ 
  Then the manifold $M=N(s_1, \dots, s_k)$
obtained by Dehn filling along $s_1, \dots, s_k$ is hyperbolic, and its volume satisfies
\[
  \left(1-\left(\frac{2\pi}{\lmin}\right)^2\right)^{3/2} \vol(N) \leqslant\vol(M) < \vol(N).
  \]
\end{teo}

To continue recall that for a compact oriented 3-manifold $M,$ 
 we set
 $$lTV(N)=\liminf_{r\to\infty} \frac{2\pi}{r}\log\left|TV\left(M,e^{\frac{2\pi i}{r}}\right)\right|$$ and  $$LTV(N)= \limsup_{r\to\infty} \frac{2\pi}{r}\log\left|TV\left(M,e^{\frac{2\pi i}{r}}\right)\right|.$$ 
  where $r$ runs over all odd integers. 

Our results in the previous sections give the following.

\begin{teo}\label{expo}
Let $M$ be an orientable, compact 3-manifold with empty or toroidal boundary. There is a  hyperbolic link $L_1\subset M$ such that
 $$lTV(M\setminus L_1)=LTV(M\setminus L_1)=\vol(M\setminus L_1)=2cv_8,$$

where is $c>0$ is a constant depending on $L_1.$

Furthermore, if we set $N= M\setminus L_1$ then for any link $L$ in $N$ we have 
  $$lTV(N\setminus L)\geqslant  2cv_8>0.$$
    \end{teo}
    
    \begin{proof} By \cite[Proposition 3.36]{CosThurston}, there is $c=c(M)>0$ and a fundamental shadow link $L'_1\subset M_c=\#^{c+1}(S^1\times S^2)$ such that 
   
    (i) The complement of $L'_1$ is hyperbolic with volume $2c v_8;$ and
    
     (ii) $M$ is obtained by Dehn filling in $M_c$ along $L'_1.$
     
    Thus there is $L_1 \subset M$ such that $M\setminus L_1$ is homeomorphic to the complement $L'_1	$ in $M_c.$
    Now the first part of the theorem follows since $M_c \setminus L'_1$ satisfy the Turaev-Viro invariants volume conjecture.
    
    To see the second part of the claim note that if $L\subset N$ is any link in  $N= M\setminus L_1$  then $N$ is obtained from $N\setminus L$ by Dehn filling. Thus by \cite[Theorem 5.3]{DK}
    we have $lTV( N\setminus L)\geqslant lTV(N)$ and the conclusion follows.
     \end{proof}
     
 \begin{dfn} We will refer to  $N= M\setminus L_1$ in the statement of Theorem \ref{expo} as  a complement of a fundamental shadow link in $M.$
 \end{dfn}

Combining  Theorem \ref{expo}  and Theorem \ref{Thm:VolChange}  gives the following which gives  Theorem \ref{Dehnvol}  stated in the Introduction.

\begin{teo}\label{TVineq}
Let $M$ be a compact 3-manifold with empty or toroidal boundary. Then there is a hyperbolic link complement $N=M\setminus L_1$
such that $M$ is obtained by Dehn filling on $N$ and we have
\[
  \alpha(\lmin) \ lTV(N) \leqslant \vol(M) < lTV(N).
  \]
  
Here  $\alpha(x)= \left(1-\left(\frac{2\pi}{x}\right)^2\right)^{3/2} $ if $x>2\pi$ , and  $\alpha(x)=0$ if $x<2\pi.$
\end{teo}

\begin{proof}
Let $M$ be an orientable, compact 3-manifold with empty or toroidal boundary and such that the interior of $M$ admits a complete hyperbolic structure.
By Theorem \ref{expo} there is a hyperbolic link complement $N\subset M$ such that
$lTV(N)=LTV(N)=\vol(N)$ and $M$ is obtained by Dehn filling along some or all the cusps of $N,$ i.e. $M=N(s_1, \dots, s_k).$ The conclusion follows by Theorem \ref{Thm:VolChange} 
\end{proof}

Note that $ \alpha(\lmin)>0,$ unless $\lmin < 2\pi$ and that $\alpha({\lmin})$ approaches 1 as $\lmin \to \infty.$
The theorem says that the volume of $M$ is approximated by the Turaev-Viro invariants of a certain sub-manifold of $M.$ It is known 
\cite{thurston:notes} that as $\lmin \to \infty$  we have $\vol(M)\to  \vol(N),$
and by Theorem \ref{TVineq}
as $\lmin \to \infty$ we also have $\vol(M)\to  lTV(N),$ which  
 is consistent with Conjecture \ref{volconj}.
In fact, by  Conjecture \ref{volconj}
one should expect a 2-sided inequality using the Turaev-Viro invariants of $M$ itself rather than these of a submanifold $N.$ In this direction, we have an one sided inequality
given by the following.

\begin{cor}  Let  $M=N(s_1, \dots, s_k)$ a 3-manifold obtained by Dehn filing on a
 fundamental shadow link complement $N.$ If
$\lmin > 2\pi,$ then $M$ is hyperbolic and we have
$$LTV(M)\  \leqslant \  B({\lmin}) \  \vol(M), $$
where $B({\lmin})$ is a function that approaches 1 as $\lmin \to \infty.$
\end{cor}
\begin{proof}  By Theorem \ref{main}, we have
$lTV(N)=LTV(N)=\vol (N).$ 
Since $\lmin > 2\pi,$ Theorem \ref{Thm:VolChange}  applies to give 
$$\left(1-\left(\frac{2\pi}{\lmin}\right)^2\right)^{3/2}   \     \vol (N) \leqslant \vol(M).$$

By  \cite[Corollary 5.3]{DK}, we have
$$LTV(M)\  \leqslant LTV(N)=\vol (N) \leqslant \left(1-\left(\frac{2\pi}{\lmin}\right)^2\right)^{-3/2} \vol(M).$$
Setting  $B({\lmin}) =\left(1-\left(\frac{2\pi}{\lmin}\right)^2\right)^{-3/2}$ we have the desired result.
\end{proof}

 \subsection{Application to the AMU conjecture}\label{secamu}
Theorem \ref{expo} says that if $N$ is the complement of a fundamental shadow link in  $M,$  then  for every link $L\subset N$ the invariants $TV_r( N\setminus L)$ grow exponentially with respect to $r.$
 As shown in \cite{DK:AMU} the exponential growth property has applications to the AMU conjecture\,\cite{ AMU}.  
 To give details, 
 for a compact  orientable surface of genus $g$ and $n$ boundary components, say
$\Sigma_{g, n},$ let $\mathrm{Mod}(\Sigma_{g, n})$ denote its  mapping class group.

\begin{dfn}For a mapping class $f\in \mathrm{Mod}(\Sigma_{g, n}),$ let $M(f)$ denote the mapping torus of $f.$ We  say that $f$  has a non-trivial pseudo-Anosov part if the toroidal decomposition of
$M(f)$ contains hyperbolic pieces; or equivalently if the Gromov norm of $M(f)$ is non-zero.
\end{dfn}

Recall that  $I_r$ is  the set $\{0,1,\dots,r-2\}.$ Given a coloring 
$col$ of the  components of $\partial \Sigma_{g,n}$ by elements of $I_r,$ by  \cite{BHMV2}, there is  a finite dimensional $\C$-vector space $V_r(\Sigma_{g,},col)$ as well as a projective representation
$$\rho_{r,col} : \mathrm{Mod}(\Sigma_{g,n}) \rightarrow \mathbb{P}\mathrm{Aut}(V_r(\Sigma_{g, n},col)).$$

The following statement is known as the AMU conjecture.

\begin{cnj}\label{AMU}  \cite{AMU} { Let $f \in  \mathrm{Mod}(\Sigma_{g, n})$ be a mapping class. If $f$ contains a pseudo-Anosov part, then for any big enough  $r$ there is a choice of colors $col$ of the components of $\partial \Sigma$ such that $\rho_{r,col}(\phi)$ has infinite order.}
\end{cnj}

In \cite{AMU} Andersen, Masbaum and Ueno  verified the conjecture for $\Sigma_{0,4}.$ Later, 
Santharoubane proved it for $\Sigma_{1,1}$  \cite{San12}  and 
 Egsgaard and Jorgensen   \cite{EgsJorgr} and Santharoubane \cite{San17} gave partial results for pseudo-Anosov maps on $\Sigma_{0,2n}.$ 
 For $g\geqslant 2,$  the first examples of mappings  classes that satisfy the AMU conjecture, were given by March\'e and Santharoubane in \cite{MarSan} and the first construction that
 leads to infinitely many (independent) examples in each genus was  given by  Detcherry and Kalfagianni \cite{DK:AMU}.

Let $M$ be a closed orientable 3-manifold with empty or toroidal boundary. Recall that a link $J$ in $M$ is called fibered if the complement $M\setminus J$
is homeomorphic to the mapping torus of a mapping class $f \in  \mathrm{Mod}(\Sigma_{g, n}),$ for some $n,g\geqslant0.$ The mapping class  $f$ is called the monodromy of the  fibration.
We note that if the first Betti number of $M\setminus J$ is bigger than one then then it can fiber over $S^1$ in infinitely many different ways; that is we have infinitely many non-conjugate
mapping classes realized as monodromies of some fibration $M\setminus  J \longrightarrow S^1$ \cite[\S3]{Thurstonorm}.

In \cite{DK:AMU} the authors  show that if we have $ lTV(M(f))>0$ for the mapping torus of a class  $f \in  \mathrm{Mod}(\Sigma_{g,n}),$ then $f$ satisfies the AMU conjecture.
On the other hand $ lTV(M(f))>0$ implies, by \cite{DK}, that $M(f)$ has non-zero Gromov norm and thus $f$ contains a pseudo-Anosov part.
In \cite{DK:AMU} the authors give explicit constructions of mapping classes that satisfy these conditions.  The examples constructed in \cite{DK:AMU}  are all realized as monodromies of fibered links
in $S^3.$ Theorem \ref{expo} provides infinite families of manifolds with toroidal boundary and with Turaev-Viro invariants having exponential growth. By passing to the doubles $DN$ we obtain closed 3-manifolds with $lTV(DN)>0.$
 Any mapping class that is realized as a monodromy of a fibered link in some $N$ or $DN$ satisfies the AMU conjecture. 

Let ${\mathcal M}$ denote the set of all 3-manifolds $N$ that are complements of fundamental shadow links in a orientable 3-manifolds with empty or toroidal boundary  and their doubles $DN.$ We have the following.

 \begin{teo:moreAMU}   Given $M\in {\mathcal M}$ and a (possibly empty) link $L\subset M,$ there is  a knot $K\subset M$ such that the link $K\cup L$ is fibered in $M.$
 Furthermore, the  monodromy of any fibration of $M\setminus (K\cup L)$ is a mapping class that 
 satisfies the AMU conjecture.
 \end{teo:moreAMU}
 \begin{proof} Let $M$ and $L$ be as above and let $L'$ a link in $S^3$ so that $M$ is obtained by integral Dehn surgery on all or some of the components of $L'.$ Note, in particular, that if $M$ is a fundamental shadow link  complement   then $L'$ will contain the link corresponding to the fundamental shadow link $L_1$ in $S^3.$ Slightly abusing the notation we will also use $L$ to denote the link in $L$ in $S^3$ corresponding 
to $L.$
 
 By Stallings \cite[Theorem 2]{Stallings},  we can find a knot $K\subset S^3$ so that $J= K\cup L\cup L'$ is a fibered link in $S^3.$ Furthermore, we have have the following.

\begin{enumerate}
\item  The knot $K$ can be chosen  so that the linking numbers of $K$ with the components of $ L\cup L'$ are arbitrary; that is matching any pre-chosen collection of integers.

\item The link $J$ is represented as a closed braid (a homogeneous braid in fact) and the fiber, say $F_K,$  of $S^3\setminus J$ is the natural Seifert surface associated to the closed braid projection.
The reader is referred to \cite{DK:AMU} for a refinement of Stallings construction that produces hyperbolic fibered links and for pictures of the fiber surface.
\end{enumerate}
 
 The components of $L'$ are equipped with the framings needed to recover $M$ from $S^3\setminus J$ by Dehn filling. On the other hand, the Seifert surface $F_K$ also defines a natural
 framing on $J:$ defined by the linking number of $J$ with a push-off of it on $F_K$ in the direction of the inward normal vector. The surface framing changes as we change the linking numbers of $K$ with the components of
 $J.$ Since we are allowed to chose these numbers to be arbitrary, by re-choosing $K,$ we can pick  $K$ so the framings defined on the components of $L'$ by the fiber $F_K$ agree with the framings of the surgery
 needed to recover $M.$ Then, the surgery caps off some components of $K$ with disks and also produces a fibered manifold. That is $M\setminus (K\cup L)$ will fiber over $S^1$ with fiber the surface obtained by $F_K$ by capping the components of $\partial F_K$ corresponding to $L'.$  By the discussion in the paragraph before the statement of the theorem,  the monodromy of such a fibration gives a mapping class
 that contains non-trivial pseudo-Anosov pieces and satisfies the AMU conjecture.
 \end{proof}

Theorem \ref{moreAMU} can be used to construct an abundance of mapping classes that satisfy the AMU conjecture. In particular, working with fibered knots in the closed manifolds of  ${\mathcal M}$ we can construct classes in $ \mathrm{Mod}(\Sigma_{g, 1}).$  This approach is developed in \cite{DK:cosets} where Detcherry and Kalfagianni show that for every closed oriented 3-manifold $M,$
 and $g$ a sufficiently large integer,  $\mathrm{Mod}(\Sigma_{g,1})$ contains a coset of an abelian subgroup of rank $\lfloor \frac{g}{2}\rfloor,$ consisting of pseudo-Anosov monodromies of fibered knots  in $M.$ Furthermore, they prove a similar result for rank two free cosets of $\mathrm{Mod}(\Sigma_{g,1}).$

Note that  all the manifolds in  ${\mathcal M}$  have Gromov norm at least $4v_8,$ but there exist fibered links of smaller Gromov norm.  It follows that there are  mapping classes in  $\mathrm{Mod}(\Sigma_{g, n}),$ $n\neq 0,$ that do not appear as monodromies of fibered links in any manifold in  ${\mathcal M}.$
We finish the subsection with the following.

\begin{que} Which mapping classes are realized by Theorem
\ref{moreAMU}?
\end{que}


\appendix
\section{Appendix}

The following theorem was originally proved by Costantino in \cite{C} for quantum $6j$-symbols evaluated at the root of unity $q=e^{\frac{\pi i}{r}},$ which is different from the one $q=e^{\frac{2\pi i}{r}}$ we considered in this paper. The main difference between the two cases is the following: For certain argument to work, some technical constrains have to be put on the sequence of $6$-tuples. It turns out that at $q=e^{\frac{\pi i}{r}}$ as considered in \cite{C} the $6$-tuples satisfying these technical  constrains never satisfy the admissibility conditions and, as a consequence, the ``evaluation'' of the quantum $6j$-symbols has to be modified; but at  $q=e^{\frac{2\pi i}{r}}$ the set of $6$-tuples satisfying both the  technical constrains and the admissibility conditions is non-empty, and those $6$-tuples are exactly the ones that give dihedral angles of ideal or hyperideal truncated tetrahedra. In this Appendix, we include a proof of the result at the root  $q=e^{\frac{2\pi i}{r}}$ for the interested readers. A similar result can also be found in \cite{Chen-Murakami}.

\begin{teo}[\cite{C}]\label{Vol}
Let $\{(n_1^{(r)},\dots,n_6^{(r)})\}$ be a sequence of 
$6$-tuples such that
\begin{enumerate}[(1)]
\item $0\leqslant Q_j-T_i\leqslant \frac{r-2}{2}$ for $i=1,\dots,4$ and $j=1,2,3,$ and
\item $\frac{r-2}{2}\leqslant T_i\leqslant r-2$ for $i=1,\dots, 4.$  
\end{enumerate}
Let $\theta_i=\lim_{r\rightarrow\infty}\frac{2\pi n_i^{(r)}}{r}$
and let $\alpha_i=|\pi-\theta_i|.$
Then
\begin{enumerate}[(1)]
\item  $\alpha_1,\dots,\alpha_6$ are the dihedral angles of an ideal or a hyperideal hyperbolic tetrahedron $\Delta,$ and
\item as $r$ runs over all the odd integers 
$$\lim_{r\to\infty}\frac{2\pi}{r}\log \Bigg|\bigg|\begin{array}{ccc}n_1^{(r)} & n_2^{(r)} & n_3^{(r)} \\n_4^{(r)} & n_5^{(r)} & n_6^{(r)} \\\end{array} \bigg|_{q=e^{\frac{2\pi i}{r}}}\Bigg|=Vol(\Delta).$$
\end{enumerate}
\end{teo}

\begin{proof} (1). By Bao and Bonahon\,\cite{BB}, six positive numbers $\alpha_1,\dots,\alpha_6$ are the dihedral angles of an ideal or a hyperideal tetrahedron if and only if around each vertex, $\alpha_i+\alpha_j+\alpha_k\leqslant \pi.$ The given conditions imply that around each vertex, $2\pi\leqslant \theta_i+\theta_j+\theta_k\leqslant 4\pi$ and $0\leqslant \theta_i+\theta_j-\theta_k\leqslant 2\pi.$ Depending on whether $\theta_i$ lies in $[0,\pi]$ or $[\pi,2\pi],$ these conditions correspond exactly  to the Bao-Bonahon conditions.

(2). By Lemma \ref{stima},  we have 
\begin{equation*}
\begin{split}
&\lim_{r\to\infty}\frac{2\pi}{r}\log\left|\Delta(n_i,n_j,n_k)\right|\\=&-\frac{1}{2}\Lambda\big(\frac{\theta_1+\theta_2-\theta_3}{2}\big)-\frac{1}{2}\Lambda\big(\frac{\theta_2+\theta_3-\theta_1}{2}
\big)\\
&-\frac{1}{2}\Lambda\big(\frac{\theta_3+\theta_1-\theta_2}{2}\big)+\frac{1}{2}\Lambda\big(\frac{\theta_1+\theta_2+\theta_3}{2}\big).
\end{split}
\end{equation*}

Next, we study the asymptotics of $$\sum_{z=\max(T_i)}^{\min(Q_j)}\frac{(-1)^z[z+1]!}{\prod_{i=1}^4[z-T_i]!\prod_{j=1}^3[Q_j-z]!}.$$
Let 
$$S_z=\frac{(-1)^z[z+1]!}{\prod_{i=1}^4[z-T_i]!\prod_{j=1}^3[Q_j-z]!}.$$
If $\lim_{r\to\infty}\frac{2\pi z}{r}=Z,$ then by Lemma \ref{stima} we have
$$\lim_{r\to\infty}\frac{2\pi}{r}\log|S_{z}|=\sum_{i=1}^4\Lambda(Z-U_i)+\sum_{j=1}^3\Lambda(V_j-Z)-\Lambda(Z).$$
The strategy is to show that all $S_z$'s for $z$ in between $\max(T_i)$ and $\min(Q_j)$ have the same sign so the growth rate of the sum is determined by that of the largest term. 

Since $T_i\leqslant z$ and $z\leqslant Q_j,$ and by the assumption that $Q_j-T_i\leqslant \frac{r-2}{2},$ we have $0\leqslant z-T_i\leqslant \frac{r-2}{2}$ and  $\leqslant Q_j-z\leqslant \frac{r-2}{2}$ for all $i=1,\dots, 4$ and $j=1,2,3.$
Hence 
$$0\leqslant \frac{2\pi(z-T_i)}{r}\leqslant\pi\quad\text{and}\quad 0\leqslant \frac{2\pi(Q_j-z)}{r}\leqslant\pi.$$
Also, by the assumption that $T_i\leqslant \frac{r-2}{2},$ $z\geqslant T_i,$ we have $\frac{r-2}{2}\leqslant z.$ Since $[z+1]!=0$ when $z>r-2,$ we can assume that $z\leqslant r-2.$  Hence
$$\pi\leqslant \frac{2\pi z}{r}\leqslant 2\pi.$$
As a consequence, we have 
$$\frac{S_{z}}{S_{z-1}}=-\frac{\sin{\frac{2\pi(z+1)}{r}}\sin{\frac{2\pi(Q_1-z+1)}{r}}\sin{\frac{2\pi(Q_2-z+1)}{r}}\sin{\frac{2\pi(Q_3-z+1)}{r}}}{\sin{\frac{2\pi(z-T_1)}{r}}\sin{\frac{2\pi(z-T_2)}{r}}\sin{\frac{2\pi(z-T_3)}{r}}\sin{\frac{2\pi(z-T_4)}{r}}}>0,$$
and hence all the $S_z$'s have the same sign.

Next we show that the function $F:(\max(U_i), \min(V_j,2\pi))\rightarrow \mathbb R$ defined by
$$F(Z)=\sum_{i=1}^4\Lambda(Z-U_i)+\sum_{j=1}^3\Lambda(V_j-Z)-\Lambda(Z)$$
has a unique maximum point $Z_0.$   Indeed, by a direct computation, one has
$$F'(Z)=\log\bigg(\frac{\sin(2\pi-Z)\sin(V_1-Z)\sin(V_2-Z)\sin(V_3-Z)}{\sin(Z-U_1)\sin(Z-U_2)\sin(Z-U_3)\sin(Z-U_4)}\bigg),$$
and
$$F''(Z)=-\sum_{i=1}^4\cot(Z-U_i)-\sum_{j=1}^3\cot(V_j-Z)-\cot(2\pi-Z).$$
Here we recall the fact that if $\alpha$ and $\beta$ are two real numbers in $(0,\pi)$ with $\alpha+\beta<\pi,$ then $\cot(\alpha)+\cot(\beta)>0.$ Now since $(2\pi-Z)+(Z-U_4)=2\pi-U_4\in(0,\pi)$ and $ (V_i-Z)+(Z-U_i)=V_i-U_i\in(0,\pi)$ for $i=1,2,3,$  we have $F''(Z)<0,$ and hence $F'(Z)$ is strictly decreasing and $F(Z)$ is strictly concave down. 
Together with 
$$\lim_{Z\to\max(U_i)}F'(Z)=+\infty\quad\text{and}\quad\lim_{Z\to\min(V_j,2\pi)}F'(V)=-\infty,$$ 
we conclude that there is a unique $Z_0\in(\max(U_i), \min(V_j,2\pi))$ such that $F'(Z_0)=0.$ By the concavity of $F,$ $Z_0$ is the unique maximum point of $F.$ 

Now for each sequence $z^{(r)}$ with $\lim_{r\to\infty}\frac{2\pi z^{(r)}}{r}=Z,$ by Lemma \ref{stima} one has 
$$|S_{z^{(r)}}|=\exp\Big\{\frac{r}{2\pi}F(Z)+O(\log(r))\Big\}\leqslant \exp\Big\{\frac{r}{2\pi}F(Z_0)+C\log(r)\Big\}.$$
Since all the $S_z$'s have the same sign, we have
$$\bigg|\sum_{z=\max(T_i)}^{\min(Q_j)}S_z\bigg|\leqslant \big(\min(Q_j,r-2)-\max(T_i)\big)\exp\Big(\frac{r}{2\pi}F(Z_0)+C\log(r)\Big),$$
and hence
\begin{equation*}
\begin{split}
&\limsup_{r\to\infty}\frac{2\pi}{r}\log\bigg|\sum_{z=\max(T_i)}^{\min(Q_j)}S_z\bigg|\\
\leqslant &\lim_{r\to\infty}\frac{2\pi}{r}\log\bigg\{\big(\min(Q_j,r-2)-\max(T_i)\big)\exp\Big(\frac{r}{2\pi}F(Z_0)+C\log(r)\Big)\bigg\}\\
=&F(Z_0).
\end{split}
\end{equation*}

On the other hand, let $z^{(r)}$ be a sequence such that 
$$\lim_{r\to\infty}\frac{2\pi z^{(r)}}{r}=Z_0.$$ Then by Lemma \ref{stima}
$$\lim_{r\to\infty}\frac{2\pi}{r}\log|S_{z^{(r)}}|=F(Z_0).$$
Again since all the $S_z$'s have the same sign, we have
$$\bigg|\sum_{z=\max(T_i)}^{\min(Q_j)}S_z\bigg|>S_{z^{(r)}},$$ and hence
$$\liminf_{r\to\infty}\frac{2\pi}{r}\log\bigg|\sum_{z=\max(T_i)}^{\min(Q_j)}S_k\bigg|\geqslant \lim_{r\to\infty}\frac{2\pi}{r}\log|S_{z^{(r)}}|=F(Z_0).$$
Therefore, we have 
$$\lim_{r\to\infty}\frac{2\pi}{r}\log\bigg|\sum_{z=\max(T_i)}^{\min(Q_j)}S_z\bigg|=F(Z_0),$$
and
\begin{equation*}
\begin{split}
&\lim_{r\to\infty}\frac{2\pi}{r}\log \Bigg|\bigg|\begin{array}{ccc}n_1^{(r)} & n_2^{(r)} & n_3^{(r)} \\n_4^{(r)} & n_5^{(r)} & n_6^{(r)} \\\end{array} \bigg|_{q=e^{\frac{2\pi i}{r}}}\Bigg|\\
=&-\frac{1}{2}\sum_{i,j}\Lambda(V_j-U_i)+\frac{1}{2}\sum_i\Lambda(U_i)+F(Z_0).
\end{split}
\end{equation*}

Then as argued in \cite{C}, by the Murakami, Yano and Ushijima formula \cite[Theorems 1 and 2]{MY}, \cite[Theorem 1.1]{vol} and their symmetry, if $\alpha_1,\dots,\alpha_6$ are the dihedral angles of a hyperideal tetrahedron $\Delta$ and  $\theta_i=\pi\pm\alpha_i,$ then the right hand side is exactly the volume of $\Delta.$
\end{proof}

\bibliographystyle{hamsplain}
\bibliography{biblio03-07}
\noindent

 \noindent 
Giulio Belletti,
Scuola Normale Superiore,
Pisa, Italy\ \ 
(giulio.belletti@sns.it)\\

\noindent 
Renaud Detcherry,
Max Planck Institute for Mathematics, Vivatsgasse 7
53111 Bonn, Germany\   \
(detcherry@mpi-bonn.mpg.de)
\\

 \noindent 
Effstratia Kalfagianni,
Department of Mathematics, Michigan State University,
East Lansing, MI 48824, USA\ \ 
(kalfagia@math.msu.edu)\\

\noindent
Tian Yang,
Department of Mathematics,  Texas A\&M University,
College Station, TX 77843, USA\ \
(tianyang@math.tamu.edu)
\end{document}